\documentclass[12pt]{amsart}
\usepackage{amsmath}	
\usepackage{amssymb}
\usepackage{tikz}
\usetikzlibrary{calc}

\newtheorem{theorem}{Theorem}[section]

\newtheorem{lemma}[theorem]{Lemma}
\newtheorem{proposition}[theorem]{Proposition}
\newtheorem{corollary}[theorem]{Corollary}

\theoremstyle{remark}

\theoremstyle{definition}
\newtheorem{definition}[theorem]{Definition}

\numberwithin{equation}{section}

\newcommand{\et}{\quad\mbox{and}\quad}

\newcommand{\bN}{\mathbb{N}}
\newcommand{\bP}{\mathbb{P}}
\newcommand{\bQ}{\mathbb{Q}}
\newcommand{\bR}{\mathbb{R}}
\newcommand{\bZ}{\mathbb{Z}}
\newcommand{\cC}{{\mathcal{C}}}

\newcommand{\cE}{{\mathcal{E}}}

\newcommand{\cS}{{\mathcal{S}}}

\newcommand{\dist}{\mathrm{dist}}

\newcommand{\hlambda}{\hat{\lambda}}
\newcommand{\htau}{\hat{\tau}}

\newcommand{\Pbar}{\bar{P}}%\overline{P}}
\newcommand{\proj}{\mathrm{proj}}

\newcommand{\Sigmabar}{\overline{\Sigma}}
\newcommand{\tbigwedge}{{\textstyle{\bigwedge}}}

\newcommand{\tM}{\tilde{M}}

\newcommand{\tP}{\tilde{P}}

\newcommand{\tuP}{\tilde{\uP}}
\newcommand{\tuR}{\tilde{\uR}}

\newcommand{\ua}{\mathbf{a}}
\newcommand{\uA}{\mathbf{A}}

\newcommand{\ue}{\mathbf{e}}
\newcommand{\uE}{\mathbf{E}}
\newcommand{\uL}{\mathbf{L}}
\newcommand{\up}{\mathbf{p}}
\newcommand{\uP}{\mathbf{P}}

\newcommand{\uR}{\mathbf{R}}
\newcommand{\uu}{\mathbf{u}}
\newcommand{\uv}{\mathbf{v}}
\newcommand{\uw}{\mathbf{w}}
\newcommand{\ux}{\mathbf{x}}

\newcommand{\uy}{\mathbf{y}}

\newcommand{\uz}{\mathbf{z}}

\newcommand{\disp}{\displaystyle}

\newcommand{\vol}{\mathrm{vol}}

\begin{document}

\baselineskip=15.2pt%20pt

\title[Parametric geometry of numbers]
{On Schmidt and Summerer\\ parametric geometry of numbers}
\author{Damien ROY}
\address{
   D\'epartement de Math\'ematiques\\
   Universit\'e d'Ottawa\\
   585 King Edward\\
   Ottawa, Ontario K1N 6N5, Canada}
\email{droy@uottawa.ca}
\thanks{Work partially supported by NSERC}

\begin{abstract}
Recently, W.~M.~Schmidt and L.~Summerer introduced a new
theory which allowed them to recover the main known inequalities relating
the usual exponents of Diophantine approximation to a point in $\bR^n$,
and to discover new ones.  They first note that these exponents 
can be computed in terms of the successive minima of a parametric 
family of convex bodies attached to the given point.  
Then they prove that the $n$-tuple of these successive
minima can in turn be approximated up to bounded difference by a function
from a certain class.  In this paper, we show that the same is true within a
smaller and simpler class of functions which we call rigid systems.
We also show that conversely, given a rigid system, there exists a point in
$\bR^n$ whose associated family of convex bodies has successive minima
which approximate that rigid system up to bounded difference.  As a consequence,
the problem of describing the joint spectrum of a family of exponents of Diophantine
approximation is reduced to combinatorial analysis.
\end{abstract}

\maketitle

\section{Introduction}
 \label{intro}

In two recent outstanding papers \cite{SS2009} and
\cite{SS2013a}, W.~M.~Schmidt and L.~Summerer study
the joint behavior of the $n$ successive minima of certain
one parameter families of convex bodies in $\bR^n$,
as a function of the parameter.  Then, they show how
their results can be used to recover important
inequalities relating standard exponents of Diophantine
approximation attached to points in $\bR^n$, and they
find new ones. The goal of this paper is to simplify
and to complete some aspects of their theory.  It also aims at
promoting their wonderful idea which, the author is convinced,
will have a major impact in Diophantine approximation as it provides
a new simple way of thinking about problems of simultaneous
approximation.

Let $n\ge 2$ be an integer, let $\ux\cdot\uy$ denote
the standard scalar product of vectors $\ux,\uy\in\bR^n$, and
let $\|\ux\|=(\ux\cdot\ux)^{1/2}$ denote the Euclidean
norm of $\ux$.  Up to re-scaling, the families of convex
bodies considered by Schmidt and Summerer are dual to
\[
 \cC_\uu(e^q)
 :=\{\ux\in\bR^n\,;\,\|\ux\|\le 1,\ |\ux\cdot\uu|\le e^{-q}\}
 \quad
 (q\ge 0),
\]
where $\uu$ is a given unit vector in $\bR^n$.  However, working
with these alternative families brings little difference and,
for our purpose, is more convenient.  For each $j=1,\dots,n$
and each $q\ge 0$, we denote by $L_{\uu,j}(q)$ the smallest
real number $L\ge 0$ such that $e^L\cC_\uu(e^q)$ contains at
least $j$ linearly independent points of $\bZ^n$.  Then,
we group these successive minima into a single map
$\uL_\uu\colon[0,\infty)\to \bR^n$ by setting
\[
 \uL_\uu(q) = (L_{\uu,1}(q),\dots,L_{\uu,n}(q))
 \quad
 (q\ge 0).
\]
Schmidt and Summerer establish many properties of this map.
We will recall them in the next section. Here we simply
mention that each component
$L_{\uu,j}\colon[0,\infty)\to\bR$ is continuous and piecewise
linear with slopes $0$ and $1$, which means that its right
derivative is $0$ or $1$ at each point $q\ge 0$, and
the same for its left derivative at each point $q>0$.
Equivalently, the graph of $L_{\uu,j}$ is a connected polygon
composed of line segments of slopes $0$ and $1$.  We also
have $0\le L_{\uu,1}(q)\le \cdots\le L_{\uu,n}(q)$ for each
$q\ge 0$, which means that $\uL_\uu$ takes values in the set
\[
 \Delta_n=\{(x_1,\dots,x_n)\in\bR^n\,;\, x_1\le \cdots\le x_n\}
\]
of monotone increasing $n$-tuples of real numbers.

Schmidt and Summerer show that each function $\uL_\uu$ can
be approximated up to bounded difference by functions
from a certain class (see in the next section).  One of our
main results is that the same property holds within a simpler
class.  To describe it, we follow
\cite[\S3]{SS2013a}, and define the \emph{combined graph} of
a set of real valued functions on an interval $I$ to
be the union of their graphs in $I\times\bR$.  For a
function $\uP=(P_1,\dots,P_n)\colon [c,\infty)\to \Delta_n$,
and a subinterval $I$ of $[c,\infty)$, we define the
\emph{combined graph of $\uP$ above $I$} to be the combined
graph of its components $P_1,\dots,P_n$ restricted to
$I$.  If $P$ is continuous and if the real numbers
$q\ge c$ at which $P_1(q),\dots,P_n(q)$ are not all distinct
form a discrete subset of $[c,\infty)$, then the map
$\uP$ is uniquely determined by its combined graph over the
full interval $[c,\infty)$.  We also denote by
$\Phi_n\colon\bR^n\to\Delta_n$ the continuous map which lists
the coordinates of a point in monotone increasing order.
We can now introduce our basic combinatorial object.

\begin{definition}
\label{intro:def:canvas}
Let $\delta\in(0,\infty)$ and let $s\in\bN^*\cup\{\infty\}
=\{\infty,1,2,3,\dots\}$.
A \emph{canvas with mesh $\delta$ and cardinality $s$} in
$\bR^n$ is a triple consisting of a sequence of points
$(\ua^{(i)})_{0\le i <s}$ in $\Delta_n$ together with
two sequences of integers $(k_i)_{0\le i <s}$ and
$(\ell_i)_{0\le i <s}$ of the same cardinality $s$
such that, for each index $i$ with $0\le i<s$,
\begin{itemize}
 \item[(C1)] the coordinates $(a^{(i)}_1,\dots,a^{(i)}_n)$
 of $\ua^{(i)}$ form a strictly increasing sequence of
 positive integer multiples of $\delta$,
 \smallskip
 \item[(C2)] we have $1\le k_0\le \ell_0 = n$ and $1\le k_i<\ell_i\le n$ if $i\ge 1$,
 \smallskip
 \item[(C3)] if $i+1<s$, then \ $k_i\le\ell_{i+1}$,
     \
     $a^{(i)}_{\ell_{i+1}} +\delta \le a^{(i+1)}_{\ell_{i+1}}$ \
     and
     \[
     (a^{(i)}_1,\dots,\widehat{a^{(i)}_{k_i}},\dots,a^{(i)}_n)
     =
     (a^{(i+1)}_1,\dots,\widehat{a^{(i+1)}_{\ell_{i+1}}},\dots,a^{(i+1)}_n)
     \]
     where the hat on a coordinate means that it is omitted.
\end{itemize}
\end{definition}

Thus, in such a sequence $(\ua^{(i)})_{0\le i <s}$, each point
$\ua^{(i+1)}$ with $i+1<s$ is obtained from the preceding point
$\ua^{(i)}$ by replacing one of its coordinates by a larger
multiple of $\delta$, different from all other coordinates of
$\ua^{(i)}$, and then by re-ordering the new $n$-tuple.
In particular, this sequence uniquely
determines the sequence $(k_i)_{0\le i <s}$ up to
its last term $k_{s-1}$ if $s<\infty$, and the full sequence
$(\ell_i)_{0\le i <s}$ since $\ell_0=n$.
When $s$ is finite, we also define $\ell_s=n$ and
\[
 \ua^{(s)}=(a^{(s)}_1,\dots,a^{(s)}_{n-1},\infty)
 \in \bR^{n-1}\times\{\infty\}
\]
so that (C3) holds for $i=s-1$.

\begin{definition}
 \label{intro:def:systems}
To each canvas of mesh $\delta>0$ as in Definition \ref{intro:def:canvas},
we associate the function $\uP\colon[q_0,\infty)\to\Delta_n$ given by
\[
 \uP(q)
 =\Phi_n\big(a^{(i)}_1,\dots,\widehat{a^{(i)}_{k_i}},\dots,a^{(i)}_n,
         a^{(i)}_{k_i}+q-q_i\big)
 \quad
 (0\le i<s, \ q_i\le q<q_{i+1}),
\]
where $q_i=a_1^{(i)}+\cdots+a_n^{(i)}$ $(0\le i<s)$ and $q_s=\infty$
if $s<\infty$.  We say that such a function is a \emph{rigid
$n$-system with mesh $\delta$} and that $(q_i)_{0\le i<s}$ is its
sequence of \emph{switch numbers}.
\end{definition}

Since $a^{(i)}_{k_i}+q_{i+1}-q_i=a^{(i+1)}_{\ell_{i+1}}$ when $i+1<s$,
such a map $\uP$ is continuous.  Its combined graph over an
interval $[q_i,q_{i+1})$ with $0\le i<s$ consists of $n-1$ horizontal
half-open line segments and one half-open line segment of slope 1.
Their left end-points are the points $(q_i,a_j^{(i)})$
$(1\le j\le n)$ and, if $i+1<s$, their right end-points are
$(q_{i+1},a_j^{(i+1)})$ $(1\le j\le n)$.  In this context, the
condition (C2), imposed on the canvas, translates into
the fact that, for each index $i$ with
$1\le i<s$, the straight line extending the line segment
of slope $1$ over $[q_i,q_{i+1})$ lies to the right of the
straight line which extends the line segment of slope $1$
over the preceding interval $[q_{i-1},q_{i})$.
Figure \ref{fig1} illustrates this by showing the combined
graph of a rigid $5$-system with mesh $1$ attached to the
canvas $\{(1,2,4,5,8),\,(1,2,4,7,8),\,(1,4,5,7,8)\}$
of cardinality $s=3$ with $k_2=1$.

\begin{figure}[h]
\label{fig1}
     \begin{tikzpicture}[scale=0.5]
       \node[draw,circle,inner sep=1pt,fill] at (0,1) {};
       \node[draw,circle,inner sep=1pt,fill] at (0,2) {};
       \node[draw,circle,inner sep=1pt,fill] at (0,4) {};
       \node[draw,circle,inner sep=1pt,fill] at (0,5) {};
       \node[draw,circle,inner sep=1pt,fill] at (0,8) {};
       \draw[dashed] (0, 8) -- (0, 0) node[below]{$q_{0}=20$};
       \node[draw,circle,inner sep=1pt,fill] at (3,1) {};
       \node[draw,circle,inner sep=1pt,fill] at (3,2) {};
       \node[draw,circle,inner sep=1pt,fill] at (3,4) {};
       \node[draw,circle,inner sep=1pt,fill] at (3,7) {};
       \node[draw,circle,inner sep=1pt,fill] at (3,8) {};
       \draw[dashed] (3, 8) -- (3, 0) node[below]{$q_{1}=22$};
       \draw[thick] (0,1) -- (3,1);
       \draw[thick] (0,2) -- (3,2);
       \draw[thick] (0,4) -- (3,4);
       \draw[thick] (0,5) -- (3,7);
       \draw[thick] (0,8) -- (3,8);
       \node[draw,circle,inner sep=1pt,fill] at (7.5,1) {};
       \node[draw,circle,inner sep=1pt,fill] at (7.5,4) {};
       \node[draw,circle,inner sep=1pt,fill] at (7.5,5) {};
       \node[draw,circle,inner sep=1pt,fill] at (7.5,7) {};
       \node[draw,circle,inner sep=1pt,fill] at (7.5,8) {};
       \draw[dashed] (7.5, 8) -- (7.5, 0) node[below]{$q_{2}=25$};
       \draw[thick] (3,1) -- (7.5,1);
       \draw[thick] (3,2) -- (7.5,5);
       \draw[thick] (3,4) -- (7.5,4);
       \draw[thick] (3,7) -- (7.5,7);
       \draw[thick] (3,8) -- (7.5,8);
       \draw[thick] (7.5,1) -- (19.5,9);
       \draw[thick] (7.5,4) -- (19.5,4);
       \draw[thick] (7.5,5) -- (19.5,5);
       \draw[thick] (7.5,7) -- (19.5,7);
       \draw[thick] (7.5,8) -- (19.5,8);
       \end{tikzpicture}
\caption{The combined graph of a rigid $5$-system.}
\end{figure}
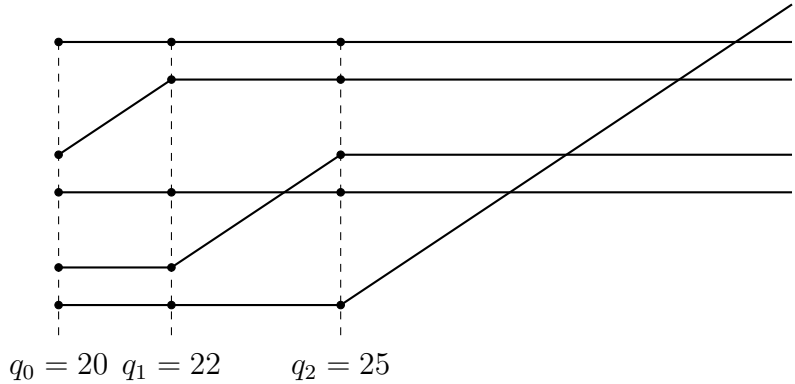

The main goal of this paper is to prove the following result.

\begin{theorem}
 \label{intro:thm}
Let $n\ge 2$ be an integer and let $\delta\in(0,\infty)$.
For each unit vector $\uu$ of $\bR^n$, there exists a rigid
system $\uP\colon[q_0,\infty)\to\Delta_n$ with mesh
$\delta$ such that $\uL_\uu-\uP$ is bounded on
$[q_0,\infty)$.  Conversely, for each rigid system
$\uP\colon[q_0,\infty)\to\Delta_n$ with mesh $\delta$,
there exists a unit vector $\uu$ in $\bR^n$
such that $\uL_\uu-\uP$ is bounded on $[q_0,\infty)$.
\end{theorem}

To mention but one application, recall that, to each unit vector
$\uu$ in $\bR^n$, one attaches several
exponents of Diophantine approximation which measure how well
$\uu$ can be approximated by rational vector subspaces of $\bR^n$
of a given dimension $d$ (see \cite{BL2010,La2009b, Sc1967}).  For $d=n-1$,
one uses $\tau(\uu)$, respectively $\htau(\uu)$, defined as the
supremum of all real numbers $\tau>0$ for which the system of
inequalities
\[
 \|\ux\| \le X \et |\ux\cdot\uu|\le X^{-\tau}
\]
admits a non-zero solution $\ux$ in $\bZ^n$ for arbitrarily large
values of $X$, respectively for all sufficiently large values of $X$.
For dimension $d=1$, one uses the dual exponents $\lambda(\uu)$,
respectively $\hlambda(\uu)$, defined as the supremum of all
$\lambda>0$ such that
\[
 \|\ux\| \le X \et \|\ux\wedge\uu\|\le X^{-\lambda}
\]
admits a non-zero solution $\ux$ in $\bZ^n$ for arbitrarily large
values of $X$, respectively for all sufficiently large values of $X$,
where $\|\ux\wedge\uu\|$ represents the norm of $\ux\wedge\uu$
for the Euclidean structure of $\tbigwedge^2\bR^n$ inherited from
$\bR^n$ (see the next section).  With this notation, Theorem
\ref{intro:thm} admits the following consequence.

\begin{corollary}
 \label{intro:cor}
Let $n\in\bN$ with $n\ge 2$ and let $\delta>0$. The map
$\theta\colon\bR^4\to\bR^4$ given by
\[
 \theta (\tau,\hat\tau,\hat\lambda,\lambda)
 \,=\,
 \left(\frac{1}{\tau+1},\, \frac{1}{\hat\tau+1},\,
   \frac{\hat\lambda}{\hat\lambda+1},\,
   \frac{\lambda}{\lambda+1}\right)
\]
establishes a bijection between the set of quadruples
$\big( \tau(\uu), \htau(\uu), \hlambda(\uu), \lambda(\uu) \big)$
where $\uu$ runs through all unit vectors of $\bR^n$ with
$\bQ$-linearly independent coordinates, and the set of quadruples
\begin{equation}
 \label{intro:cor:eq}
\left(\disp
  \liminf_{q\to\infty}\frac{P_1(q)}{q},\
  \limsup_{q\to\infty}\frac{P_1(q)}{q},\
  \liminf_{q\to\infty}\frac{P_n(q)}{q},\
  \limsup_{q\to\infty}\frac{P_n(q)}{q}
 \right)
\end{equation}
where $\uP=(P_1,\dots,P_n)$ runs through all rigid
$n$-systems with mesh $\delta$ for which $P_1$ is unbounded.
\end{corollary}

The proof is clear based on the general philosophy developed by
Schmidt and Summerer in their above-mentioned papers.  Namely,
if a unit vector $\uu$ and a rigid $n$-system $\uP=(P_1,\dots,P_n)$
with mesh $\delta$ are such that the difference $\uL_\uu-\uP$ is bounded,
then
\[
 \liminf_{q\to\infty}\frac{P_1(q)}{q}
 = \liminf_{q\to\infty}\frac{L_{\uu,1}(q)}{q}
 = \frac{1}{\tau(\uu)+1}
\]
with similar equalities for the three other components of the
quadruple \eqref{intro:cor:eq}.  Moreover, in that case, $\uu$ has
$\bQ$-linearly independent coordinates if and only if $P_1$ is unbounded.

The above corollary could easily be extended to deal with all the
intermediate exponents of Schmidt and Laurent \cite{La2009b,Sc1967},
thereby solving a conjecture of Schmidt and Summerer in \cite[\S4]{SS2013a}.
In the present context, the latter authors show that $\theta$ maps
injectively the quadruples $(\tau(\uu),\htau(\uu),\hlambda(\uu),\lambda(\uu))$
to those of the form \eqref{intro:cor:eq} where $\uP$ runs through the
larger set of what they call \emph{proper $(n,\gamma)$-systems} (see
Section \ref{subsec:SS} below).  From this, they recover the
celebrated Khintchine's and Jarn\'{\i}k's transference principles
\cite{Ja1938,Kh1926a,Kh1926b} as well as more recent results of
Bugeaud, Laurent and Moshchevitin from \cite{BL2010,La2009,La2009b,Mo2012b}.
They also prove new results \cite{SS2009,SS2013a,SS2013b}.
Recently, Laurent gave in \cite{La2009} a complete description
of the joint spectrum of $(\tau,\htau,\hlambda,\lambda)$ in
dimension $n=3$.  For larger dimension $n\ge 4$, the problem is open
and the above corollary reduces it to combinatorial analysis.  Note however
that the present study does not apply to the more general exponents
introduced by German in \cite{Ge2012} (see also \cite{Sc1967}).

In the next section, we recast in our setting the result
of Schmidt and Summerer which approximates the maps $\uL_\uu$
by $(n,\gamma)$-systems.  We also present there the intermediate
results of geometry of number which are involved in the proof.
In Section \ref{sec:obs}, we use the same results to analyze
the situation where a rigid system with large mesh is,
in comparison, well approximated by the map $\uL_\uu$ for
some unit vector $\uu$.  The results that we obtain there
complement Theorem \ref{intro:thm}. They also motivate the
constructions of Section \ref{sec:inv} by which we prove the
second assertion of Theorem \ref{intro:thm} for rigid
systems with sufficiently large mesh (cf.\ \cite{R_preprint}
for a special case of this construction). The last piece of the
puzzle is provided by Sections \ref{sec:red} and
\ref{sec:app} which construct an approximation to an arbitrary
$(n,\gamma)$-system by a rigid system with given sufficiently
large mesh.  The process is first to modify the $(n,\gamma)$-system
to make what we call a \emph{reduced system} (Section \ref{sec:red}),
then to approximate the resulting system by a step function and
finally to construct a canvas out of this data (Section \ref{sec:app}).
The proof of our main Theorem \ref{intro:thm} follows in Section
\ref{sec:proof}.  The reader may go there directly to get a precise
idea on the role of all intermediate results.

%%%%%%%%%%%%%%%%%%%%%%%%%%%%%%%%%%%%%%%
%
%  The theory of Schmidt and Summerer
%
%%%%%%%%%%%%%%%%%%%%%%%%%%%%%%%%%%%%%%%

\section{The theory of Schmidt and Summerer}
 \label{sec:conv}

We start by recalling the basic notions and results from geometry of
numbers that we will need throughout this paper.  In few places, we
provide a short argument in order to be able to specify the constants
involved.  We also present the central result of Schmidt and Summerer
theory and its proof, which we adapt to our slightly different (dual)
context.  We hope that this will help the reader firstly because this
makes our account relatively self contained, and secondly because, in
the next section, we use the same notions and intermediate results to
gather information about the inverse problem raised by this theorem.
All results stated below are thus either classical or due to Schmidt
and Summerer.

Let $V$ be a real Euclidean vector space of finite dimension $N\ge 1$.
We use the following standard terminology (see \cite{GL1987}).
By a \emph{convex body} of $V$, we mean a compact convex neighborhood
of $0$, stable under multiplication by $-1$.  By a \emph{lattice}
$\Lambda$ of $V$, we mean a discrete subgroup of $V$ of rank $N$.  Its
\emph{co-volume} is the volume of the parallelepiped spanned
by a basis of $\Lambda$ or, equivalently, the volume of $V/\Lambda$.

Suppose that $\cC$ is a convex body of $V$, and $\Lambda$
a lattice of $V$.  For each $j=1,\dots,N$, we define the
\emph{$j$-th minimum of $\cC$ with respect to $\Lambda$}, denoted
$\lambda_j(\cC)$, to be the smallest
real number $\lambda>0$ such that $\lambda\cC$ contains at least
$j$ elements of $\Lambda$ which are linearly independent over $\bR$.
Although this notation does not refer to the lattice $\Lambda$,
this should not cause any ambiguity since, in all situations that
we consider, the underlying lattice will be clear from the context.
For each $\ux\in V$, we further define $\lambda_\ux(\cC)$, also denoted
$\lambda(\ux,\cC)$, to be the smallest real number $\lambda\ge0$
such that $\ux\in\lambda\cC$.  Then, there exist elements $\ux_1,\dots,
\ux_N$ of $\Lambda$ which are linearly independent over $\bR$ and
satisfy $\lambda(\ux_j,\cC)=\lambda_j(\cC)$ for $j=1,\dots,N$.
The function from $V$ to $\bR$ mapping a point $\ux\in V$ to
$\lambda_\ux(\cC)$ is called the \emph{distance function}
of $\cC$.  The notation $\lambda_\ux(\cC)$ with the point $\ux$ in
index stresses the fact that, in the theory of Schmidt and Summerer,
the convex body $\cC$ is varying.

\subsection{A general family of convex bodies}
\label{subsec:gen}
Let $V$ be as above and let $\Lambda$ be a lattice of $V$ with
co-volume $1$.  We choose a decomposition of $V$ into an orthogonal
sum $V=U\perp W$ of two vector subspaces $U$ and $W$ with
$W\neq 0$, and put $K=\dim_\bR(W)$.  Motivated by
\cite{SS2013a}, we consider the family of convex bodies of $V$
given by
\[
 \cC(Q)
  := \{\ux\in V\,;\,
      \text{$\|\ux\|\le 1$ and $\|\proj_W(\ux)\|\le Q^{-1}$}\}
 \quad
 (Q\ge 1),
\]
where $\proj_W$ stands for the orthogonal projection on $W$.
For each $j=1,\dots,N$, we define a function $L_{j}\colon
[0,\infty)\to \bR$ by
\[
 L_{j}(q)=\log \lambda_j(\cC(e^q))
 \quad
 (q\ge 0).
\]
Clearly, we have $L_1(q)\le \cdots\le L_N(q)$ for each $q\ge 0$,
and so we get a map $\uL\colon [0,\infty)\to\Delta_N$ by putting
\[
 \uL(q)=(L_{1}(q),\dots,L_{N}(q))
 \quad
 (q\ge 0).
\]
Then, Minkowski's second convex body theorem yields the following
result.

\begin{lemma}
\label{conv:lem:L+...+L}
For each $q\ge0$, we have $|L_1(q)+\cdots+L_N(q)-Kq|\le N\log(N)$.
\end{lemma}

\begin{proof}
Let $Q\ge 1$.  According to Minkowski's second convex
body theorem, we have
\[
 2^N/N!
  \le \lambda_1(\cC(Q))\cdots\lambda_N(\cC(Q))\vol(\cC(Q))
  \le 2^N.
\]
To estimate the volume of $\cC(Q)$, we choose an orthonormal
basis $(\ue_1,\dots,\ue_K)$ of $W$, extend it to an
orthonormal basis $(\ue_1,\dots,\ue_N)$ of $V$, and form the
parallelepiped $P$ given by
\[
 P=\Big\{\ux\in V\,;\,
 \max_{1\le j\le K} |\ux\cdot\ue_j|\le Q^{-1}
 \text{ and }
 \max_{K<j\le N} |\ux\cdot\ue_j|\le 1 \Big\}.
\]
Since $N^{-1}P\subset \cC(Q) \subset P$, we obtain
$(2/N)^NQ^{-K} \le \vol(\cC(Q))\le 2^NQ^{-K}$.  Thus
\[
 (N!)^{-1}
 \le \lambda_1(\cC(Q))\cdots\lambda_N(\cC(Q))Q^{-K}
 \le N^N
\]
and the conclusion follows by taking logarithms.
\end{proof}

\subsection{Trajectories of points and combined graphs}
\label{subsec:traj}
Let the notation be as in \S\ref{subsec:gen}.
For each $\ux\in V$ and each $Q\ge 1$, we find
\[%begin{equation}
% \label{conv:eq:lambda_x}
 \lambda_\ux(\cC(Q))
  = \lambda(\ux,\cC(Q))
  = \max\{\|\ux\|,\, Q\|\proj_W(\ux)\|\}.
\]%end{equation}
When $\ux\neq 0$, this number is positive, and so we obtain
a function $L_\ux\colon[0,\infty)\to \bR$ by putting
\begin{equation}
 \label{conv:eq:L_x}
 L_\ux(q) := \log\lambda_\ux(\cC(e^q))
           = \max\{\log\|\ux\|, q+\log\|\proj_W(\ux)\|\}
 \quad \text{($q\ge 0$).}
\end{equation}
It is continuous and piecewise linear. If $\proj_W(\ux)=0$,
it is constant equal to $\log \|\ux\|$.  Otherwise, it has
slope $0$ and then $1$.  We also note that, if $\ux$ and $\uy$ are
linearly dependent non-zero elements of $V$, then
$L_\ux$ and $L_\uy$ differ by a constant.  In particular,
they have the same derivative $L'_\ux(q)=L'_\uy(q)$ at each
point $q>0$ at which they are differentiable.

For a fixed non-zero $\ux$ in $V$, the function $L_\ux$ describes
the position of $\ux$ with respect to the varying family of convex bodies
$\cC(e^q)$.  For this reason, we call its graph the \emph{trajectory}
of $\ux$.  Explicitly, this is the set $\{(q,L_\ux(q))\,;\,0\le q\}$.
The inclusion
\[
 \{ (q,L_j(q)) \,;\, 0\le q,\ 1\le j\le N \}
 \subseteq
 \{ (q,L_\ux(q)) \,;\, 0\le q,\ \ux\in\Lambda\setminus\{0\} \}
\]
may thus be expressed by saying that the \emph{combined graph}
of $L_1,\dots,L_n$ is contained in the combined graph of the functions
$L_\ux$ with $\ux\in\Lambda\setminus\{0\}$, namely the union of
the trajectories of these points.  The goal is, in a sense, to
compare these two sets.  Clearly, we have \[
 L_1(q)=\inf\{L_\ux(q)\,;\,\ux\in\Lambda\setminus\{0\}\}
 \quad (q\ge 0).
\]
The next crucial result, due to Schmidt and Summerer, goes a step further.

\begin{lemma}
\label{conv:lem:L1=L2}
The functions $L_1,\dots,L_N$ are continuous and piecewise
linear with slopes $0$ and $1$.  At each point $q>0$ where
$L_1$ changes slope from $1$ to $0$, we have $L_1(q)=L_2(q)$.
\end{lemma}

\begin{proof}
Since $\cC(e^{q_1}) \supseteq \cC(e^{q_2}) \supseteq
e^{q_1-q_2}\cC(e^{q_1})$ for each choice of $q_2\ge q_1\ge 0$,
we have
\[
 L_j(q_1)
 \le L_j(q_2)
 \le L_j(q_1)+(q_2-q_1)
 \quad
 \text{($1\le j\le N$, $0\le q_1\le q_2$).}
\]
Thus $L_1,\dots,L_N$ are continuous. The inequality $\log\|\ux\|
\le \lambda_\ux(q)$, valid for $q\ge 0$ and $\ux\in V\setminus\{0\}$,
shows that any bounded region of $[0,\infty)\times\bR$ meets
only finitely trajectories of points $\ux\in\Lambda\setminus\{0\}$.
As the latter cover the graphs of $L_1,\dots,L_N$ and consist
of at most two line segments of slope $0$ and $1$, we conclude that
$L_1,\dots,L_N$ are piecewise linear with slopes $0$ and $1$.
Finally, suppose that $L_1$ changes slope
from $1$ to $0$ at a point $q>0$. Then, there exist
$\epsilon>0$ and two non-zero points $\ux$
and $\uy$ in $\Lambda$ such that
\[
 L_1(t)=\begin{cases}
   L_\ux(t) &\text{for $t\in[q-\epsilon,q]$,}\\
   L_\uy(t) &\text{for $t\in[q,q+\epsilon]$.}
         \end{cases}
\]
This implies that $L'_\ux(q)=1$ and $L'_\uy(q)=0$.  So, the points
$\ux$ and $\uy$ are linearly independent.  As they
both belong to $\exp(L_1(q))\cC(e^q)$, we conclude that
$L_2(q)=L_1(q)$.
\end{proof}

\subsection{The main family of convex bodies}
\label{subsec:main_family}
Fix an integer $n\ge 2$ and a unit vector $\uu$ of $\bR^n$.
We apply the preceding considerations to the decomposition
\[
 \bR^n=U\perp W
 \quad \text{where} \quad
 W=\langle\uu\rangle_\bR
 \et
 U:=W^\perp =\{\ux\in\bR^n\,;\, \ux\cdot\uu=0\},
\]
using the standard integer lattice $\Lambda=\bZ^n$.
Since $\|\proj_W(\ux)\|=|\ux\cdot\uu|$ for each $\ux\in\bR^n$,
this gives rise to the family of convex bodies
\[
 \cC_\uu(Q)
  := \{\ux\in \bR^n\,;\,
      \text{$\|\ux\|\le 1$ and $|\ux\cdot\uu|\le Q^{-1}$}\}
 \quad
 (Q\ge 1),
\]
and its associated map $\uL_\uu=(L_{\uu,1},\dots,L_{\uu,n})
\colon [0,\infty) \to \Delta_n$ where
$L_{\uu,j}(q)=\log\lambda_j(\cC_\uu(e^q))$.
By Lemmas \ref{conv:lem:L+...+L} and \ref{conv:lem:L1=L2},
the functions $L_{\uu,1},\dots,L_{\uu,n}$ are continuous,
piecewise linear with slopes $0$ and $1$, and they satisfy
\begin{equation}
 \label{conv:eq:L1+...+Ln}
 |L_{\uu,1}(q)+\cdots+L_{\uu,n}(q)-q|\le n\log(n)
 \quad
 (q\ge 0).
\end{equation}
In particular, these functions are monotone increasing.
By \eqref{conv:eq:L_x}, the trajectory of a non-zero
point $\ux\in\bZ^n$, with respect to the family of convex
bodies $\cC_\uu$, is the graph of the function
$L_\ux\colon[0,\infty)\to\bR$ given by
\begin{equation}
 \label{conv:eq:L(x,q)}
 L_\ux(q) = L(\ux,q) := \max\{\log\|\ux\|, q+\log|\ux\cdot\uu|\}
 \quad \text{($q\ge 0$).}
\end{equation}
Equivalently, we note that $\lambda(\ux,\cC_\uu(Q))=
\max\{\|\ux\|,\,Q|\ux\cdot\uu|\}$ \, ($Q\ge 1$).

\subsection{Families of pseudo-compound convex bodies}
\label{subsec:pseudo}
Let $\uu$ and $U$ be as in \S\ref{subsec:main_family}.
We fix an integer $k\in\{1,\dots,n\}$ and endow the vector
space $V=\tbigwedge^k\bR^n$ with the unique structure
of Euclidean space such that, for any orthonormal basis
$(\ue_1,\dots,\ue_n)$ of $\bR^n$, the products
$\ue_{j_1}\wedge\cdots\wedge\ue_{j_k}$ with
$1\le j_1<\cdots<j_k\le n$ form an orthonormal basis of
$\tbigwedge^k\bR^n$.  We also define $\Lambda =
\tbigwedge^k\bZ^n$ to be the lattice of co-volume $1$ spanned
by all products $\ux_1\wedge\cdots\wedge\ux_k$ with
$\ux_1,\dots,\ux_k\in\bZ^n$.  We have the orthogonal
sum decomposition
\[
 \tbigwedge^k\bR^n = U^{(k)} \perp W^{(k)}
 \quad \text{where}\quad
 U^{(k)}=\tbigwedge^kU
 \et
 W^{(k)}=\tbigwedge^{k-1}U\wedge\langle\uu\rangle_\bR.
\]
In accordance with the general construction of
\S\ref{subsec:gen}, we set
\[
 N:=\dim_\bR(\tbigwedge^k\bR^n)=\binom{n}{k}
 \et
 K:=\dim_\bR(W^{(k)})=\binom{n-1}{k-1}
\]
and, for each $Q\ge 1$, we define
\[
 \cC_\uu^{(k)}(Q)
 = \big\{\omega\in\tbigwedge^k\bR^n \,;\, \|\omega\|\le 1
     \ \text{and}\ \|\proj_{W^{(k)}}(\omega)\|\le Q^{-1}\big\}.
\]
We also form the associated map
$\uL^{(k)}_\uu=(L^{(k)}_{\uu,1},\dots,L^{(k)}_{\uu,N})
\colon [0,\infty) \to \Delta_N$ given by
\[
 L_{\uu,j}^{(k)}(q)=\log\lambda_j(\cC^{(k)}_\uu(e^q))
 \quad
 (q\ge 0,\ 1\le j\le N).
\]
By Lemmas \ref{conv:lem:L+...+L} and \ref{conv:lem:L1=L2},
its components are continuous piecewise linear functions
with slopes $0$ and $1$, and they satisfy
\begin{equation}
 \label{conv:eq:L1+...+LN}
 |L^{(k)}_{\uu,1}(q)+\cdots+L^{(k)}_{\uu,N}(q)-Kq|\le N\log(N)
 \quad
 (q\ge 0).
\end{equation}
Moreover we have the equality $L_{\uu,2}^{(k)}(q)=L_{\uu,1}^{(k)}(q)$
at each point $q>0$ where $L_{\uu,1}^{(k)}$ changes slope from $1$ to $0$.

When $k=1$, we have $\tbigwedge^1\bR^n=\bR^n$,
$W^{(1)}=\langle\uu\rangle_\bR$, $\cC_\uu^{(1)}(Q)=\cC_\uu(Q)$ $(Q\ge 1)$,
and $\uL^{(1)}_\uu=\uL_\uu$.  In general, we use $\cC_\uu^{(k)}(Q)$
as an approximation of the $k$-th compound convex
body of $\cC_\uu(Q)$, namely the convex hull of the exterior products
of $k$ elements of $\cC_\uu(Q)$ (see \cite[Lemma 3]{BL2010}).
The next lemma shows that this compound body is contained in
$k\cC^{(k)}_\uu(Q)$.

\begin{lemma}
 \label{conv:lemma:C^k(Q)}
Let $Q\ge 1$ and let $\ux_1,\dots,\ux_k\in\cC_\uu(Q)$.  Then
$\ux_1\wedge\cdots\wedge\ux_k\in k\cC^{(k)}_\uu(Q)$.
\end{lemma}

\begin{proof}
For $j=1,\dots,k$, define $\uy_j=\proj_U(\ux_j)$ so that
$\ux_j=\uy_j+(\ux_j\cdot\uu)\uu$.  Upon writing $\omega =
\ux_1\wedge\cdots\wedge\ux_k$, we find
\[
 \omega
 = \uy_1\wedge\cdots\wedge\uy_k
   + \sum_{j=1}^k (-1)^{j+k} (\ux_j\cdot\uu)
         \uy_1\wedge\cdots\widehat{\,\uy_j}\wedge
                    \cdots\wedge\uy_k\wedge\uu\,.
\]
Thus the sum on the right is $\proj_{W^{(k)}}(\omega)$
and therefore
\[
 \|\proj_{W^{(k)}}(\omega)\|
   \le \sum_{j=1}^k |\ux_j\cdot\uu|\,
         \|\uy_1\|\cdots\widehat{\|\uy_j\|}\cdots\|\uy_k\|
   \le kQ^{-1}.
\]
As $\|\omega\|\le \|\ux_1\|\cdots\|\ux_k\|\le 1 \le k$, we conclude that
$\omega\in k\cC^{(k)}_\uu(Q)$.
\end{proof}

According to \eqref{conv:eq:L_x}, the trajectory of a non-zero point
$\omega\in\tbigwedge^k\bZ^n$ with respect to the family $\cC_\uu^{(k)}$
is the graph of the function $L_\omega\colon[0,\infty)\to\bR$
given by
\begin{equation}
 \label{conv:eq:L(k)}
 L_\omega(q)
  = L(\omega,q)
 := \max\{\log\|\ux\|, q+\log\|\proj_{W^{(k)}}(\omega)\|\}
 \quad \text{($q\ge 0$).}
\end{equation}
With this notation, the previous lemma generalizes as follows.

\begin{lemma}
 \label{conv:lemma:L(k)}
Let $\uy_1,\dots,\uy_k$ be linearly independent elements of $\bZ^n$.
We have
\[
 L(\uy_1\wedge\cdots\wedge\uy_k,q)
 \le L(\uy_1,q)+\cdots+L(\uy_k,q)+\log(k)
 \quad
 (q\ge 0).
\]
\end{lemma}

\begin{proof}
Fix a choice of $q\ge 0$ and set $\omega=\uy_1\wedge\cdots
\wedge\uy_k$. For each $j=1,\dots,k$, we have
$\lambda_j^{-1}\uy_j\in\cC_\uu(e^q)$ where $\lambda_j:=\exp(L(\uy_j,q))$.
Thus the product($\lambda_1\dots\lambda_k)^{-1}\omega$
belongs to $k\cC^{(k)}_\uu(e^q)$, and so
$L(\omega,q) \le \log(k\lambda_1\dots\lambda_k)$.
\end{proof}

In the case where $k=n$, we have $U^{(n)}=0$. Then, our
convex bodies are balls
\[
 \cC_\uu^{(n)}(Q)
 = \big\{\omega\in\tbigwedge^n\bR^n \,;\, \|\omega\|\le Q^{-1}\big\}
 \quad
 (Q\ge 1).
\]
So we find $L(\omega,q)=q+\log\|\omega\|$ for
any non-zero $\omega\in\tbigwedge^n\bR^n$ and any $q\ge 0$.
In particular, this gives $L^{(n)}_{\uu,1}(q)=q$ ($q\ge 0$),
and the preceding lemma admits the following consequence.

\begin{lemma}
 \label{conv:lemma:L(n)}
Let $\uy_1,\dots,\uy_n$ be linearly independent elements of $\bZ^n$.
We have
\[
 \log\|\uy_1\wedge\cdots\wedge\uy_n\|
 \le L(\uy_1,q)+\cdots+L(\uy_n,q)-q+\log(n)
 \quad
 (q\ge 0).
\]
\end{lemma}

The next estimate essentially goes back to Mahler \cite{Ma1955}
(see also \cite[\S15.2]{GL1987}).

\begin{lemma}
\label{conv:lemma:M(k)-L(k)}
For each $q\ge 0$, let $(S_{\uu,1}^{(k)}(q),\dots,
S_{\uu,N}^{(k)}(q))$ denote the sequence of all sums
$L_{\uu,j_1}(q)+\cdots+L_{\uu,j_k}(q)$ with $1\le j_1<\cdots<j_k\le n$
written in monotone increasing order.  Then, we have
\[
 -\log(n)\le S_{\uu,j}^{(k)}(q)-L_{\uu,j}^{(k)}(q) \le 2^nn\log(n)
 \quad
 (q\ge 0,\ 1\le j\le N).
\]
\end{lemma}

\begin{proof}
Fix a choice of $q\ge 0$ and choose linearly independent
points $\uy_1,\dots,\uy_n$ of $\bZ^n$ which realize the
successive minima of $\cC_\uu(e^q)$ in the sense that
$L(\uy_j,q)=L_{\uu,j}(q)$ for $j=1,\dots,n$.  Then,
denote by $(\mu_1,\dots,\mu_N)$ the set of numbers
$L(\uy_{j_1}\wedge\cdots\wedge\uy_{j_k},q)$ arranged in
monotone increasing order (with the convention that
$1\le j_1<\cdots<j_k\le n$).  Since the products
$\uy_{j_1}\wedge\cdots\wedge\uy_{j_k}$ are linearly
independent elements of $\tbigwedge^k\bZ^n$, we have
$L^{(k)}_{\uu,j}(q)\le \mu_j$ for $j=1,\dots,N$.  By Lemma
\ref{conv:lemma:L(k)}, we also have
\[
 L(\uy_{j_1}\wedge\cdots\wedge\uy_{j_k},q)
 \le L_{\uu,j_1}(q)+\cdots+L_{\uu,j_k}(q)+\log k
 \quad
 (1\le j_1<\cdots<j_k\le n),
\]
and so $\mu_j\le S^{(k)}_{\uu,j}(q)+\log(k)$ for $j=1,\dots,N$.
Combining these two observations gives
\[
 L^{(k)}_{\uu,j}(q) \le S^{(k)}_{\uu,j}(q)+\log(k)
 \quad
 (1\le j\le N).
\]
On the other hand, the estimates \eqref{conv:eq:L1+...+Ln} and
\eqref{conv:eq:L1+...+LN} yield
\[
 \sum_{j=1}^N \big(S^{(k)}_{\uu,j}(q)+\log(k)-L^{(k)}_{\uu,j}(q)\big)
  = K\sum_{j=1}^n L_{\uu,j}(q)+N\log(k)-\sum_{j=1}^N L^{(k)}_{\uu,j}(q)
 \le c
\]
where $c=Kn\log(n)+N\log(k)+N\log(N)=Kn\log(n)+N\log(kN)$.  Since
$nK=kN\le n^k$, we have $c\le 2Kn\log(n)\le 2^nn\log(n)$, and
the conclusion follows.
\end{proof}

We simply need the following consequence of the previous lemma.

\begin{lemma}
\label{conv:lemma:LLL}
Put $c_1=2^nn\log(n)$.  For each $q\ge 0$, we have
\begin{itemize}
 \item[(i)]
 $|L_{\uu,1}^{(k)}(q)-L_{\uu,1}(q)-\cdots-L_{\uu,k}(q)| \le c_1$,
 \smallskip
 \item[(ii)]
 $|L_{\uu,2}^{(k)}(q)-L_{\uu,1}(q)-\cdots-L_{\uu,k-1}(q)-L_{\uu,k+1}(q)| \le c_1$
 \ if \ $1< k< n$,
 \smallskip
 \item[(iii)]
 $|L_{\uu,\,j}^{(n-1)}(q)+L_{\uu,\,n+1-j}(q)-q| \le c_1+n\log(n)$
 \ for \ $j=1,\dots,n$.
\end{itemize}
\end{lemma}

\begin{proof}
The first two estimates follow immediately from Lemma
\ref{conv:lemma:M(k)-L(k)} because the smallest sum
$L_{\uu,j_1}(q)+\cdots+L_{\uu,j_k}(q)$
is obtained by choosing $j_1=1,\dots,j_k=k$, while the next one,
when $1<k<n$, is obtained by choosing $j_1=1,\dots,j_{k-1}=k-1$
and $j_k=k+1$.  If $k=n-1$, then the $j$-th sum is
$S_{\uu,j}^{(n-1)}(q)=L_{\uu,1}(q)+\cdots+L_{\uu,n}(q)-L_{\uu,n+1-j}(q)$
and so (iii) follows from \eqref{conv:eq:L1+...+Ln}.
\end{proof}

\subsection{The approximation theorem of Schmidt and Summerer}
\label{subsec:SS}
Fix an integer $n\ge 2$. The following definition is extracted from
\cite[\S2]{SS2013a} and adapted to our context.

\begin{definition}
\label{Conv:def:SS}
Let $\gamma,q_0\ge 0$.  An $(n,\gamma)$-system on the half-line
$[q_0,\infty)$ is a function
$\uP=(P_1,\dots,P_n)\colon [q_0,\infty) \to \bR^n$ which
satisfies the following conditions.
\begin{itemize}
 \item[(S1)]
   $-\gamma\le P_j(q)\le P_{j+1}(q)+\gamma$
    \quad ($1\le j<n$, $q_0\le q$).
   \smallskip
 \item[(S2)]
   $P_j(q_1)\le P_j(q_2)+\gamma$
    \quad ($1\le j\le n$, $q_0\le q_1\le q_2$).
   \smallskip
 \item[(S3)]
   For $j=1,\dots,n$, the function
   $M_j:=P_1+\cdots+P_j\colon[q_0,\infty)\to\bR$ is continuous
   and piecewise linear with slopes $0$ and $1$.
  \smallskip
 \item[(S4)]
   $M_n(q)=q$ \quad ($q_0\le q$).
  \smallskip
 \item[(S5)]
   If, for $j\in\{1,\dots,n-1\}$, the function $M_j$ changes
   slope from $1$ to $0$ at a point $q>q_0$, then $P_{j+1}(q)
   \le P_j(q)+\gamma$.
\end{itemize}
\end{definition}

Up to the value of $\gamma$, the next result comes from
\cite[\S2]{SS2013a}.  It shows the fundamental importance of the
notion of $(n,\gamma)$-systems.  For the sake of completeness
and because our context is slightly different, we also recall
its proof below.

\begin{theorem}[Schmidt-Summerer, 2013]
\label{conv:thm:SS}
Let $\gamma=6n2^n\log(n)$.  For each unit vector
$\uu$ of $\bR^n$, there exists an $(n,\gamma)$-system
$\uP\colon[0,\infty)\to\bR^n$ such that
\[
 \sup_{q\ge 0}\|\uP(q)-\uL_\uu(q)\|_\infty \le \gamma.
\]
\end{theorem}

\begin{proof}
Set $M_0:=0$ and $M_k:=L_{\uu,1}^{(k)}$ for $k=1,\dots,n$.
By Lemma \ref{conv:lemma:LLL} (i), we have
\[
 |M_k(q)-L_{\uu,1}(q)-\cdots-L_{\uu,k}(q)| \le c_1
 \quad
 \text{($1\le k\le n$, $0\le q$).}
\]
Then, upon defining $P_k=M_k-M_{k-1}$ for $k=1,\dots,n$,
we obtain
\begin{equation}
 \label{SS:eq:P-L}
 |P_k(q)-L_{\uu,k}(q)| \le 2c_1
 \quad
 \text{($1\le k\le n$, $0\le q$)}
\end{equation}
So, in order to complete the proof of the theorem, it suffices
to show that $\uP:=(P_1,\dots,P_n)$ is an $(n,6c_1)$-system
on $[0,\infty)$.
The conditions (S1) and (S2) derive immediately from
\eqref{SS:eq:P-L} because for each $q\ge 0$ we have
\[
 \begin{aligned}
 P_k(q)&\ge L_{\uu,k}(q)-2c_1 \ge -2c_1,\\
 P_k(q)&\le L_{\uu,k}(q)+2c_1 \le L_{\uu,k+1}(q)+2c_1\le P_{k+1}(q)+4c_1
 \end{aligned}
 \quad
 \text{($1\le k <n$),}
\]
while for $0\le q_1\le q_2$ we find
\[
 P_k(q_1)\le L_{\uu,k}(q_1)+2c_1 \le L_{\uu,k}(q_2)+2c_1\le P_k(q_2)+4c_1
 \quad
 \text{($1\le k \le n$).}
\]
For $k=1,\dots,n$, we also have $P_1+\cdots+P_k=M_k-M_0=M_k$,
and we know that $M_k=L_{\uu,1}^{(k)}$
is continuous and piecewise linear with slopes $0$ and $1$.
Moreover, as noticed in \S\ref{subsec:pseudo}, we have $M_n(q)
=L_{\uu,1}^{(n)}(q)=q$ for each $q\ge 0$.
Thus (S3) and (S4) are automatically satisfied.  Finally, (S5)\
also holds because if, for some $k\in\{1,\dots,n-1\}$, the
function $M_k=L_{\uu,1}^{(k)}$ changes slope from $1$ to $0$ at
a point $q>0$, then $L_{\uu,2}^{(k)}(q)=L_{\uu,1}^{(k)}(q)$ and so
the estimates (i) and (ii) of Lemma \ref{conv:lemma:LLL}
yield $L_{\uu,k+1}(q)-L_{\uu,k}(q) \le 2c_1$,
which in turn implies that $P_{k+1}(q)\le P_k(q)+6c_1$.
\end{proof}

Note that the properties of an $(n,\gamma)$-system are
simpler when $\gamma=0$.  For example, the condition (S3) implies
that the components of an $(n,\gamma)$-system are continuous
piecewise linear functions with slopes $-1$, $0$ or $1$.  However,
when $\gamma=0$, it follows from (S2) that these components
are monotone increasing and so, possess only
the slopes $0$ and $1$.  Moreover, (S1) shows that an $(n,0)$-system
takes values in $\Delta_n$ while this may also fail for a general
$(n,\gamma)$-system.  A general description of $(n,0)$-systems
is given in \cite[\S3]{SS2013a}.  We conclude this section
with the following observation which describes the rigid
$n$-systems as a subset of the set of $(n,0)$-systems, leaving
its proof to the reader.

\begin{lemma}
 \label{conv:lem:systems}
Let $\delta\in(0,\infty)$.  The rigid $n$-systems of mesh $\delta$ are the
$(n,0)$-systems $(P_1,\dots,P_n)\colon [q_0,\infty)\to\bR^n$
with the property that, for $q=q_0$ and for each $q>q_0$ at which
at least one of the functions $P_1+\cdots+P_j$ ($1\le j<n$) changes
slope from $0$ to $1$, the numbers $P_1(q),\dots,P_n(q)$ are $n$
distinct multiples of $\delta$.
\end{lemma}

In particular, a rigid system $(P_1,\dots,P_n)\colon [q_0,\infty)\to\bR^n$
satisfies
\begin{equation}
 \label{conv:eq:P_1+...+P_n}
 P_1(q)+\cdots+P_n(q)=q \quad (q\ge q_0).
\end{equation}
This is an important property that we will use repeatedly.

%%%%%%%%%%%%%%%%%%%%%%%%%%%%%%%%%%%%%%%%%%%%%%%%
%
%  A special case
%
%%%%%%%%%%%%%%%%%%%%%%%%%%%%%%%%%%%%%%%%%%%%%%%%

\section{A special case}
\label{sec:obs}

In this section, we fix an integer $n\ge 2$ and assume that a rigid
$n$-system $\uP$ of large mesh is, in comparison, very closely approximated
by the map $\uL_\uu$ for some unit vector $\uu$ of $\bR^n$.  We derive
from this the existence of a sequence of $n$-tuples of integer points
with strong properties.  This partly explains and motivates the constructions
of Section \ref{sec:inv} where, given an arbitrary rigid system $\uP$
of sufficiently large mesh we construct a unit vector $\uu$ such that
$\uP-\uL_\uu$ is bounded.  Our goal is thus to prove the following complement
to Theorem \ref{intro:thm}.

\begin{theorem}
\label{obs:thm}
Let $\delta,\epsilon$ be positive real numbers with
$\delta\ge 6(n\epsilon+c_1)$ where $c_1$ is as in Lemma
\ref{conv:lemma:LLL}, and let
$\uP=(P_1,\dots,P_n)\colon[q_0,\infty)\to \bR^n$
be a rigid system of mesh $\delta$.  Suppose that
there exists a unit vector $\uu\in\bR^n$ such that
$\|\uP(q)-\uL_\uu(q)\|_\infty \le \epsilon$
for any $q\ge q_0$.  Consider the sequences
$(q_i)_{0\le i<s}$, $(k_i)_{0\le i<s}$ and $(\ell_i)_{0\le i<s}$
attached to $\uP$ as in Definition \ref{intro:def:systems},
and set $q_s=\infty$
if $s<\infty$.  Then, for each integer $i$ with $0\le i<s$,
there exists an $n$-tuple of linearly independent integer
points $(\ux^{(i)}_1,\dots,\ux^{(i)}_n)$ with the following
properties:
\begin{itemize}
 \item[1)]
  $\big\| \uP(q) - \Phi_n\big(L(\ux^{(i)}_1,q),\dots,L(\ux^{(i)}_n,q)\big)\big\|_\infty
  \le \epsilon$
  \quad $(q_i\le q<q_{i+1})$,
  \smallskip
 \item[2)]
  $(\ux^{(i)}_1,\dots,\widehat{\ux^{(i)}_{k_{i}}},\dots,\ux^{(i)}_n)
  = (\ux^{(i+1)}_1,\dots,\widehat{\ux^{(i+1)}_{\ell_{i+1}}},\dots,\ux^{(i+1)}_n)$
  \ if \ $i+1<s$,
  \smallskip
 \item[3)]
  $\big|\log\|\ux^{(i)}_j\|-P_j(q_i)\big|\le \epsilon$
  \quad $(j=1,\dots,n)$ \ if \ $i\ge 1$,
  \smallskip
 \item[4)]
  $\ux^{(i+1)}_{\ell_{i+1}}
   \in \big\langle \ux^{(i)}_1,\dots,\ux^{(i)}_{\ell_{i+1}}\big\rangle_\bR$
   \ if \ $i+1<s$,
   \smallskip
 \item[5)]
  $0 \le \log|\det(\ux^{(i)}_1,\dots,\ux^{(i)}_n)|\le n\epsilon+\log(n)$,
  \smallskip
 \item[6)]
  $\big|\log\|\ux^{(i)}_1\wedge\dots\wedge\widehat{\ux^{(i)}_{k_i}}
     \wedge\dots\wedge\ux^{(i)}_n\| - \sum_{j\neq{k_i}}\log\|\ux^{(i)}_j\| \big|
   \le n\epsilon+2c_1$ \ if \ $i\ge 1$.
\end{itemize}
\end{theorem}

The property 1) means that, over each interval $[q_i,q_{i+1})$, the
combined graph of $\uP$ is contained in an $\epsilon$-neighborhood
of the union of the trajectories of $\ux^{(i)}_1,\dots,\ux^{(i)}_n$.
Together with the hypothesis, it implies that, for each $q$ in
that interval, the latter points realize the logarithms of the
successive minima of $\cC_\uu(e^q)$ within $2\epsilon$.
Because of the specific form of the combined
graph of $\uP$, we also infer from 1) that exactly one of the points
$\ux^{(i)}_j$ must have a trajectory of slope $1$ over
$[q_i+2\epsilon,\infty)$ while all other points have trajectories
of slope $0$ over $[0,q_{i+1}-2\epsilon)$ (see Lemma \ref{obs:lem2}).
If $i+1<s$, it follows from 2) that this exceptional point must be
$\ux^{(i)}_{k_i}$.  Over the next interval $[q_{i+1},q_{i+2})$, its
trajectory is replaced by that of a new point $\ux^{(i+1)}_{\ell_{i+1}}$
while those of the other points are kept.  Figure \ref{fig2} illustrates
this on an example with $n=5$.  In this picture the solid lines represent
the combined graph of $\uP$ and the dotted lines the trajectories
of the points $\ux^{(i)}_j$.

\begin{figure}[h]
\label{fig2}
      \begin{tikzpicture}[scale=0.5]
       \node[draw,circle,inner sep=0.8pt,fill] at (0,1) {};
       \node[draw,circle,inner sep=0.8pt,fill] at (0,2) {};
       \node[draw,circle,inner sep=0.8pt,fill] at (0,4) {};
       \node[draw,circle,inner sep=0.8pt,fill] at (0,5) {};
       \node[draw,circle,inner sep=0.8pt,fill] at (0,8) {};
       \draw[dashed] (0, 8) -- (0, 0) node[below]{$q_{0}$};
       \node[draw,circle,inner sep=0.8pt,fill] at (3,1) {};
       \node[draw,circle,inner sep=0.8pt,fill] at (3,2) {};
       \node[draw,circle,inner sep=0.8pt,fill] at (3,4) {};
       \node[draw,circle,inner sep=0.8pt,fill] at (3,7) {};
       \node[draw,circle,inner sep=0.8pt,fill] at (3,8) {};
       \draw[dashed] (3, 8) -- (3, 0) node[below]{$q_{1}$};
       \draw[semithick] (0,1) -- (3,1);
       \draw[semithick] (0,2) -- (3,2);
       \draw[semithick] (0,4) -- (3,4);
       \draw[semithick] (0,5) -- (3,7);
       \draw[semithick] (0,8) -- (3,8);
       \node[draw,circle,inner sep=0.8pt,fill] at (7.5,1) {};
       \node[draw,circle,inner sep=0.8pt,fill] at (7.5,4) {};
       \node[draw,circle,inner sep=0.8pt,fill] at (7.5,5) {};
       \node[draw,circle,inner sep=0.8pt,fill] at (7.5,7) {};
       \node[draw,circle,inner sep=0.8pt,fill] at (7.5,8) {};
       \draw[dashed] (7.5, 8) -- (7.5, 0) node[below]{$q_{2}$};
       \draw[semithick] (3,1) -- (7.5,1);
       \draw[semithick] (3,2) -- (7.5,5);
       \draw[semithick] (3,4) -- (7.5,4);
       \draw[semithick] (3,7) -- (7.5,7);
       \draw[semithick] (3,8) -- (7.5,8);
       \draw[semithick] (7.5,1) -- (19.5,9);
       \draw[semithick] (7.5,4) -- (19.5,4);
       \draw[semithick] (7.5,5) -- (19.5,5);
       \draw[semithick] (7.5,7) -- (19.5,7);
       \draw[semithick] (7.5,8) -- (19.5,8);
       \draw[thick, dotted] (0,8.2) -- (20,8.2) node[right]{$\ux^{(0)}_5=\ux^{(1)}_5=\ux^{(2)}_5$};
       \draw[thick, dotted] (0,5.2) -- (0.6,5.2) -- (6.3,9) {};
       \node[right] at (6.3,9.4) {$\ux^{(0)}_4$};
       \draw[thick, dotted] (1.5,6.8) -- (20,6.8) node[right]{$\ux^{(1)}_4=\ux^{(2)}_4$};
       \draw[thick, dotted] (0,3.8) -- (0.1,3.9) -- (20,3.9) node[right]{$\ux^{(0)}_3=\ux^{(1)}_3=\ux^{(2)}_2$};
       \draw[thick, dotted] (0,2.14) -- (3.42,2.14) -- (13.69,9) {};
       \node[right] at (13.69,9.4) {$\ux^{(0)}_2=\ux^{(1)}_2$};
       \draw[thick, dotted] (6,5.16) -- (20,5.16) node[right]{$\ux^{(2)}_3$};
       \draw[thick, dotted] (0,0.8) -- (6.9,0.8) -- (19.2,9) {};
       \node[right] at (19.6,9.4) {$\ux^{(0)}_1=\ux^{(1)}_1=\ux^{(2)}_1$};
       \end{tikzpicture}
\caption{The combined graph of a rigid $5$-system and the
trajectories of integer points that approximate it within $\epsilon$.}
\end{figure}
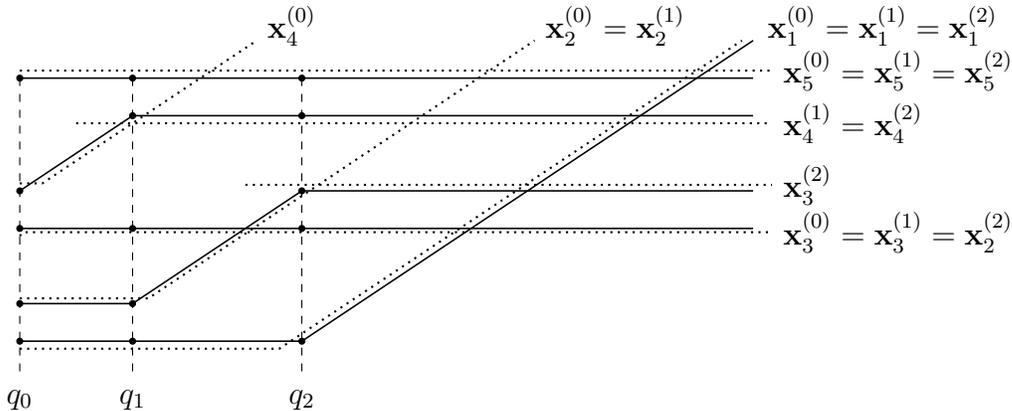

The properties 3) to 6) are useful additional information about the
points $\ux^{(i)}_j$.  For example, 5) tells us that each $n$-tuple
$(\ux^{(i)}_1,\dots,\ux^{(i)}_n)$ spans a subgroup of $\bZ^n$ with uniformly
bounded index, while 6) implies that the angles between any two of the
points $\ux^{(i)}_1,\dots,\widehat{\ux^{(i)}_{k_i}},\dots,\ux^{(i)}_n$
are bounded away from $0$.  In Section \ref{sec:inv}, dealing with
the inverse problem, we construct $n$-tuples of integer points with
stronger properties.  In particular, we request that each $n$-tuple
$(\ux^{(i)}_1,\dots,\ux^{(i)}_n)$ is a basis of $\bZ^n$ and that the points
$\ux^{(i)}_1,\dots,\widehat{\ux^{(i)}_{k_i}},\dots,\ux^{(i)}_n$ are almost
orthogonal in a sense that is defined in \S\ref{sec:dist}.

\subsection*{Preliminary observations towards the proof of Theorem \ref{obs:thm}.}
Suppose that the hypotheses of Theorem \ref{obs:thm} are fulfilled.
For each $\ux\in\bR^n\setminus\{0\}$ and each $j\in\{1,\dots,n\}$, we set
\[
 L(\ux,\infty) = \lim_{q\to\infty} L(\ux,q),\quad
 L_{\uu,j}(\infty) = \lim_{q\to\infty} L_{\uu,j}(q),\quad
 P_j(\infty) = \lim_{q\to\infty} P_j(q).
\]
We also define
\begin{align*}
 V_j(q)
  &=\big\langle
     \ux\in\bZ^n\setminus\{0\}\,;\, L(\ux,q)\le L_{\uu,j}(q)
    \big\rangle_\bR
 \quad (1\le j\le n,\ 0\le q\le \infty),\\
 V_0(q)&=0
 \quad (0\le q\le \infty).
\end{align*}
The vector space $V_j(\infty)$ is interesting only when
$L_{\uu,j}(\infty)<\infty$ or equivalently when
$P_j(\infty)<\infty$.  In that case, $V_j(\infty)$ is orthogonal
to $\uu$ and we have $L(\ux,q)=\log\|\ux\|$ \ $(q\ge 0)$
for any $\ux\in V_j(\infty)\setminus\{0\}$.  Otherwise, we have
$V_j(\infty)=\bR^n$.  We start with three observations.

\begin{lemma}
\label{obs:lem1}
Let $j\in\{1,\dots,n-1\}$.  The family of vector spaces $V_j(q)$
is constant, with dimension $j$, on any subinterval of $[q_0,\infty)$
on which $P_{j+1}(q)>P_j(q)+2\epsilon$.
\end{lemma}

\begin{proof}
Since $\|\uP(q)-\uL_\uu(q)\|_\infty\le \epsilon$ for each $q\ge q_0$,
the condition $P_{j+1}(q)>P_j(q)+2\epsilon$ implies that
$L_{\uu,j+1}(q)>L_{\uu,j}(q)$ and thus $\dim_\bR V_j(q)=j$.
For any given $q\ge q_0$ with this property, there exists a
connected neighborhood $I$ of $q$ in $[q_0,\infty)$ and a real
number $r$ such that $L_{\uu,j}(t)< r < L_{\uu,j+1}(t)$
for each $t\in I$.  Then we have
\[
 V_j(t)
  =\big\langle
    \ux\in\bZ^n\setminus\{0\}\,;\, L(\ux,t) < r
   \big\rangle_\bR
 \quad (t\in I)
\]
and so the family $V_j(t)$ is monotone decreasing on $I$.
As it has constant dimension $j$, it is therefore constant.
Thus the family $V_j(q)$ is locally constant and therefore
constant on any subinterval of $[q_0,\infty)$
on which $P_{j+1}>P_j+2\epsilon$.
\end{proof}

\begin{lemma}
\label{obs:lem2}
Let $i$ be an integer with $0\le i<s$.  Put $k=k_i$.  Then,
there exists a point $\ux\in\bZ^n$ with $\ux\notin V_{k-1}(q_i)$
such that
\begin{equation}
\label{obs:lem2:eq}
 \big|L(\ux,t)-(P_k(q_i)+\max\{0,t-q_i\})\big|\le \epsilon
 \quad
 (t\ge 0).
\end{equation}
In particular, the trajectory of $\ux$ has slope $0$ on
$[0,q_i-2\epsilon]$ if $i\ge 1$ and slope $1$ on $[q_i+2\epsilon,\infty)$.
\end{lemma}

\begin{proof}
Let $I=\{t\ge q_0\,;\, |t-q_i|<\delta-2\epsilon\}$.  If $k\ge 2$,
we have $P_k(q_i)\ge P_{k-1}(q_i)+\delta$ and so
$P_k(t) > P_{k-1}(t)+2\epsilon$ for any $t\in I$.  By Lemma
\ref{obs:lem1}, this implies that $V_{k-1}(t)=V_{k-1}(q_i)$ has
dimension $k-1$ for each $t\in I$.  This is also true if $k=1$
because then $V_{k-1}(t)=0$ for each $t\ge 0$.

Since $\dim_\bR V_{k-1}(q_i)=k-1$, there exists $\ux\in\bZ^n$
with $\ux\notin V_{k-1}(q_i)$ such that $L(\ux,q_i)=L_{\uu,k}(q_i)$.
We claim that this point satisfies the condition
\eqref{obs:lem2:eq}.  For each $t\in I$, we have $\ux\notin
V_{k-1}(t)$, and therefore
\[
 L(\ux,t)\ge L_{\uu,k}(t) \ge P_k(t)-\epsilon
 \quad
 (t\in I).
\]
As $L(\ux,t)$ has slope $0$ and $1$, we also have
\[
 L(\ux,t) \le L(\ux,q_i)+\max\{0,t-q_i\}
 \quad
 (t\ge 0).
\]
By construction, the function $P_k$ is constant on $[q_i-\delta,q_i]$
if $i\ge 1$, and has slope $1$ on $[q_i,q_i+\delta]$.  Since
$L(\ux,q_i) = L_{\uu,k}(q_i) \le P_k(q_i)+\epsilon$, the
previous inequality yields
\[
 L(\ux,t)
  \le P_k(q_i)+\max\{0,t-q_i\}+\epsilon
  = P_k(t)+\epsilon
 \quad
 (t\in I).
\]
Thus, \eqref{obs:lem2:eq} holds for each $t\in I$.
Since $\delta>4\epsilon$, this implies that $L(\ux,t)$ is constant
on $[0,q_i-2\epsilon]$ if $i\ge 1$, and that it has slope $1$
on $[q_i+2\epsilon,\infty)$.  Thus the validity of \eqref{obs:lem2:eq}
extends from $I$ to the whole interval $[0,\infty)$.
\end{proof}

\begin{lemma}
\label{obs:lem3}
Suppose that $P_j(\infty)<\infty$ for some index $j$
with $1\le j< n$.  Then $V_j(\infty)$ has dimension $j$
and there exists a real number $q\ge 0$ such that
$V_j(t)=V_j(\infty)$ for each $t\ge q$.  Moreover, there
exists a point $\ux\in\bZ^n\cap V_j(\infty)$ with
$\ux\notin V_{j-1}(\infty)$ such that
$|L(\ux,t)-P_j(\infty)|\le \epsilon$ for each $t\ge 0$.
\end{lemma}

\begin{proof}
Given a real number $q\ge 0$, there are finitely many non-zero
points $\ux\in\bZ^n$ such that $L(\ux,t)\le L_{\uu,j}(t)$ for some
$t\ge q$ because this implies that $\log\|\ux\|\le P_j(t)+\epsilon
\le P_j(\infty)+\epsilon$.  Moreover, the set of these points $\ux$
is monotone decreasing as a function of $q$.  Choose $q$ so that it
has minimal cardinality.  Then, all points $\ux$ with the above
property have $L(\ux,t)$ constant, equal to $\log\|\ux\|$, for
$t\ge q$.  Thus, $L_{\uu,j}$ is also constant on $[q,\infty)$ and
we find
\[
 V_j(t)
 = \big\langle
   \ux\in\bZ^n\setminus\{0\}\,;\,
    \ux\cdot\uu=0, \ \log\|\ux\|\le L_{\uu,j}(\infty)
   \big\rangle_\bR
 = V_j(\infty)
\]
for each $t\ge q$.  If $q$ is sufficiently large, we also have
$P_{j+1}(q)>P_j(q)+2\epsilon$ because $P_{j+1}(\infty) >
P_j(\infty)+\delta$.  Then, $V_j(\infty)=V_j(q)$ has dimension $j$
according to Lemma \ref{obs:lem1}.

By the above, we also have $\dim_\bR V_{j-1}(\infty)=j-1$ if
$j\ge 2$, because then $P_{j-1}(\infty)<\infty$.  This remains true
if $j=1$ since $V_0(\infty)=0$.  Thus $V_j(\infty)\neq
V_{j-1}(\infty)$.  Then, assuming $q$ sufficiently large so that
$V_{j-1}(q)=V_{j-1}(\infty)$, we have $V_j(q)\neq V_{j-1}(q)$ and
consequently there exists $\ux\in\bZ^n\cap V_j(q)$ with $\ux\notin V_{j-1}(q)$
such that $L(\ux,q)=L_{\uu,j}(q)$.  This point has the requested
property because $L(\ux,t)=L_{\uu,j}(\infty)$ for $t\ge 0$,
and $|L_{\uu,j}(\infty)-P_j(\infty)|\le \epsilon$.
\end{proof}

Thanks to Lemmas \ref{obs:lem1} and \ref{obs:lem3}, we note that
$\dim_\bR V_j(q_i)=j$ for each pair of integers $i$ ($\neq \infty$)
and $j$ with $0\le i\le s$ and $0\le j\le n$.

\subsection*{Construction of the points $\ux^{(i)}_j$ and proof
of Properties 1--3).}
We denote by $\cS$ the set of all maximal horizontal line segments
of positive length contained in the combined graph of $\uP$, together
with the point $(q_0,P_{k_0}(q_0))$ (which is thus the only line
segment of length $0$ in $\cS$).

To each $S\in\cS$ we associate a point $\ux_S\in\bZ^n$ in the
following way.  If $S$ is bounded, its right end-point is
$(q_i,P_{k_i}(q_i))$ for some integer $i$ with $0\le i<s$.  Then, we
choose for $\ux_S$ any point $\ux$ with the property stated in Lemma
\ref{obs:lem2}.  Otherwise, $S$ is contained in
$[0,\infty)\times \{P_j(\infty)\}$ for some index $j$ with $1\le j<n$
such that $P_j(\infty)<\infty$.  Then, we choose for $\ux_S$
any point $\ux$ with the property stated in Lemma \ref{obs:lem3}.

Finally, for each pair of integers $i$ and $j$ with $0\le i<s$ and
$1\le j\le n$, there exists a unique line segment $S$ in $\cS$
containing the point $(q_i,P_j(q_i))$ and we define $\ux_j^{(i)}
= \ux_S$.  In the degenerate case where $s<\infty$, we have
$P_j(\infty)<\infty$ for each $j=1,\dots,n-1$. Then, for those $j$,
we define $\ux_j^{(s)}=\ux_S$ where $S$ is the unique unbounded
line segment of $\cS$ contained in $[0,\infty)\times\{P_j(\infty)\}$.
By construction, these points $\ux^{(s)}_1,\dots,\ux^{(s)}_{n-1}$ are
linearly independent.

With these definitions, Property 2) in Theorem \ref{obs:thm}
is automatically satisfied.  More precisely we have
\begin{equation}
\label{obs:eq1}
 (\ux^{(i)}_1,\dots,\widehat{\ux^{(i)}_{k_{i}}},\dots,\ux^{(i)}_n)
  = (\ux^{(i+1)}_1,\dots,\widehat{\ux^{(i+1)}_{\ell_{i+1}}},\dots,\ux^{(i+1)}_n)
 \quad
 (0\le i<s)
\end{equation}
upon defining $\ell_s=n$ if $s<\infty$.  Property 1) is also clear
in view of the description of the trajectories of the points $\ux$
given by Lemmas \ref{obs:lem2} and \ref{obs:lem3}, as illustrated on
Figure \ref{fig2} for $n=5$.  Indeed, for any fixed index $i$ with
$0\le i<s$, we have
\begin{align}
 &\sup_{q_i\le t<q_{i+1}}
  \big| L(\ux_{k_i}^{(i)},t)-P_{k_i}(q_i)-(t-q_i) \big|
  \le \epsilon,
 \label{obs:eq:L-P}\\
 &\sup_{q_i\le t<q_{i+1}}
  \big| L(\ux_j^{(i)},t)-P_j(q_i) \big|
  \le \epsilon
 \qquad
 (1\le j \le n, j\neq k_i).
 \label{obs:eq:L-Pbis}
\end{align}
Since $\uP(t)=\Phi_n(P_1(q_i),\dots,\widehat{P_{k_i}(q_i)},\dots,P_n(q_i),
P_{k_i}(q_i)+t-q_i)$ for each $t\in [q_i,q_{i+1})$, and since
$\|\Phi(\up)-\Phi(\up')\|_\infty \le \|\up-\up'\|_\infty$ for any
$\up,\up'\in\bR^n$, this yields 1).  Moreover, if $i\ge 1$, we have
$|L(\ux_j^{(i)},t)-P_j(q_i)| \le \epsilon$ for each $t\in [0,q_i]$
and each $j=1,\dots,n$.  Upon setting $=0$, this yields Property 3).

\subsection*{Linear independence of the $n$-tuples $(\ux^{(i)}_1,\dots,\ux^{(i)}_n)$}
This is the most delicate part of the argument.  It will come out
of the last of the following four lemmas.

\begin{lemma}
 \label{obs:lem4}
Let $i,j$ be integers with $0\le i<s$ and $0\le j<n$.  If $P_{j+1}$
is constant on $[q_i,\infty]$, then $V_j(q_i)=V_j(\infty)$.  If $P_{j+1}$
is constant on $[q_i,q_r]$ for some integer $r$ with $i\le r<s$,
then $V_j(q_i)=V_j(q_r)$.
\end{lemma}

\begin{proof}
If $j=0$, this is clear since $V_0(q)=0$ for any $q\in[0,\infty]$.
Suppose that $j\ge 1$.  In the first case, we choose $q\ge q_i$
sufficiently large so that, by Lemma \ref{obs:lem3}, we have
$V_j(q)=V_j(\infty)$.  In the second case, we set $q=q_r$.
In both cases, we have $P_{j+1}(t)\ge P_j(t)+\delta$ for each
$t\in[q_i,q]$ and so, by Lemma \ref{obs:lem1}, we conclude that
$V_j(q_i)=V_j(q)$.
\end{proof}

\begin{lemma}
 \label{obs:lem5}
Let $i,j$ be integers with $0\le i\le s$ and $0\le j\le n$, and
let $\omega$ be a generator of the one-dimensional vector space
$\tbigwedge^jV_j(q_i)$.
\begin{itemize}
 \item[1)] If $0<i\le s$ and $j<\ell_i$, then $L(\omega,t)$ is
 constant on $[0,q_i-\delta/3]$ (understood as $[0,\infty]$ if
 $i=s$, since in that case $q_s=\infty$).
 \item[2)] If $0\le i< s$ and $j\ge k_i$, then $L(\omega,t)$
 has slope $1$ on $[q_i+\delta/3, \infty)$.
\end{itemize}
\end{lemma}

\begin{proof}
Since $V_j(q_i)$ has dimension $j$, its $j$-th exterior power
has dimension $1$.  For any other generator $\omega'$ of the latter
vector space, the difference $L(\omega,t)-L(\omega',t)$ is
constant.  So, we may assume that
$\omega=\uy_1\wedge\cdots\wedge\uy_j$ where $\uy_1,\dots,\uy_j$
are linearly independent points of $\bZ^n$ such that
$L(\uy_m,q_i)=L_{\uu,m}(q_i)$ for $m=1,\dots,j$.

If $i=s$ and $j<\ell_s=n$, then $\uy_1,\dots,\uy_j$ belong to
$U=\uu^\perp$ (because $V_j(\infty)\subseteq U$) and so
$\omega\in\tbigwedge^j U$.  By virtue of formula \eqref{conv:eq:L(k)}
with $k$ replaced by $j$, this implies that $L(\omega,t)=
\log\|\omega\|$ is constant for all $t\ge 0$.

From now on, we may therefore assume that $i<s$, and so
$q_i<\infty$.  As $\omega$ is a non-zero element of
$\tbigwedge^j\bZ^n$ we have
\[
 L(\omega,t) \ge L_{\uu,1}^{(j)}(t) \quad (t\ge 0).
\]
By Lemma \ref{conv:lemma:LLL} (i), we also have
\[
 L_{\uu,1}^{(j)}(t)
  \ge L_{\uu,1}(t)+\cdots+L_{\uu,j}(t)-c_1
  \ge M_j(t)-j\epsilon-c_1
 \quad
 (t\ge q_0),
 \]
where $M_j=P_1+\cdots+P_j$.  On the other hand, since
$L(\omega,t)$ has slopes $0$ and $1$, it satisfies
\[
 L(\omega,t) \le L(\omega,q_i)+\max\{0,t-q_i\}
 \quad (t\ge 0),
\]
while Lemma \ref{conv:lemma:L(k)} gives
\begin{align*}
 L(\omega,q_i)
  \le L(\uy_1,q_i)+\cdots+L(\uy_j,q_i)&+\log n\,
  \\
   = L_{\uu,1}(q_i)+\cdots+L_{\uu,j}(q_i)&+\log n
  \le M_j(q_i)+j\epsilon+c_1.
\end{align*}
Combining these estimates, we conclude that
\[
 |L(\omega,t)-M_j(t)|
  \le |M_j(q_i)+\max\{0,t-q_i\}-M_j(t)| + (\delta/6)
 \quad
 (t\ge q_0)
\]
since $j\epsilon+c_1\le \delta/6$.
If $0<i<s$ and $j<\ell_i$, then $M_j$ is constant on
$[q_i-\delta,q_i]$ and the latter estimate gives
$|L(\omega,t)-M_j(q_i)| \le \delta/6$ for each $t\in[q_i-\delta,q_i]$.
So, $L(\omega,t)$ is constant on $[0,q_i-\delta/3]$.
If $0\le i<s$ and $j\ge k_i$, then $M_j$ has slope $1$ on
$[q_i,q_i+\delta]$ and the same estimate yields
$|L(\omega,t)-M_j(t)| \le \delta/6$ for each $t\in[q_i,q_i+\delta]$.
Then, $L(\omega,t)$ has slope $1$ on $[q_i+\delta/3,\infty)$.
\end{proof}

\begin{lemma}
 \label{obs:lem6}
Let $i,j\in\bN$ with $0\le i< s$ and $k_i\le j< \ell_{i+1}$.
Then, $V_j(q_i)\neq V_j(q_{i+1})$.
\end{lemma}

\begin{proof}
Suppose on the contrary that $V_j(q_i) = V_j(q_{i+1})$ and let
$\omega$ be a generator of the $j$-th exterior power of this vector
space.  By Lemma \ref{obs:lem5} applied successively to the points
$q_i$ and $q_{i+1}$, the function $L(\omega,t)$ has slope $1$ on
$[q_i+\delta/3,\infty)$ and is constant on $[0,q_{i+1}-\delta/3]$.
As $q_{i+1}\ge q_i+\delta$, this is impossible.
\end{proof}

\begin{lemma}
 \label{obs:lem7}
For each pair of integers $i,j$ with $0\le i< s$ and $0\le j\le n$,
we have $V_j(q_i)=\langle \ux^{(i)}_1,\dots,\ux^{(i)}_j\rangle_\bR$.
In particular $\ux^{(i)}_1,\dots,\ux^{(i)}_n$ are linearly independent
for each index $i$ with $0\le i<s$.
\end{lemma}

\begin{proof}
The second assertion follows from the first because $V_n(q_i)=\bR^n$
for each $i$.  We prove the first assertion by induction on $j$.  For
$j=0$, it is clear since $V_0(q_i)=0$ for each $i$. Suppose that
it holds for some integer $j$ with $0\le j<n$.  Fix an arbitrary index
$i$ with $0\le i<s$.  Since $V_{j+1}(q_i)$ has dimension $j+1$ and
contains $V_j(q_i)$ as a vector subspace of dimension $j$, we are
reduced to showing that
$\ux^{(i)}_{j+1}\in V_{j+1}(q_i)\setminus V_j(q_i)$, in order to
complete the induction step.  If $j<n-1$, we find that
\[
 L(\ux^{(i)}_{j+1},q_i)
  \le P_{j+1}(q_i)+\epsilon
   < P_{j+2}(q_i)-\epsilon
  \le L_{\uu,j+2}(q_i)
\]
where the first inequality comes from \eqref{obs:eq:L-P} and
\eqref{obs:eq:L-Pbis}.  Thus, we have $\ux^{(i)}_{j+1}\in
V^{(i)}_{j+1}$ if $j<n-1$.  If $j=n-1$, the conclusion is the same
because $V_n(q_i)=\bR^n$.  So, it remains to show that
$\ux^{(i)}_{j+1}\notin V_j(q_i)$.

Let $S$ be the line segment in $\cS$ which contains the point
$(q_i,P_{j+1}(q_i))$ so that $\ux_{j+1}^{(i)}=\ux_S$.  If $P_{j+1}$
is constant on $[q_i,\infty)$, then Lemma \ref{obs:lem4} gives
$V_j(q_i)=V_j(\infty)$ and, by construction, we have $\ux_S\notin
V_j(\infty)$.  We may therefore assume the existence of a largest
integer $r$ with $i\le r<s$ such that $P_{j+1}$ is constant on
$[q_i,q_r]$.  Then the same lemma shows that $V_j(q_i)=V_j(q_r)$.
By construction, we also have $\ux_{j+1}^{(i)}=\ux_{j+1}^{(r)}=\ux_S$.
Thus, upon replacing $i$ by $r$ if necessary, we may assume that
$P_{j+1}$ is non-constant on $[q_i,q_{i+1}]$.  This means that
$k_i\le j+1\le \ell_{i+1}$.  If $j+1=k_i$, then $(q_i,P_{j+1}(q_i))$
is the right end-point of $S$ and, by construction, we have
$\ux_S\notin V_j(q_i)$.  So, we may further assume that
$k_i\le j<\ell_{i+1}$.  By Lemma \ref{obs:lem6}, this implies that
$V_j(q_i)\neq V_j(q_{i+1})$.  However, by the induction hypothesis,
we have
\begin{align*}
 V_j(q_i)
  &=\langle \ux^{(i)}_1,\dots,\ux^{(i)}_j \rangle_\bR,\\
 V_j(q_{i+1})
  &=\langle \ux^{(i+1)}_1,\dots,\ux^{(i+1)}_j \rangle_\bR
   =\langle
    \ux^{(i)}_1,\dots,\widehat{\ux^{(i)}_{k_{i}}},\dots,\ux^{(i)}_{j+1}
    \rangle_\bR
\end{align*}
where the second expression for $V_j(q_{i+1})$ comes from \eqref{obs:eq1}
(if $i+1=s$, the second formula is not part of the induction hypothesis
but follows from the definition of the points $\ux^{(s)}_1,\dots,
\ux^{(s)}_{n-1}$).  From this, we conclude that
$\ux^{(i)}_{j+1}\notin V_j(q_i)$.
\end{proof}

\subsection*{Proof of Property 4).}
Let $i\in\bN$ with $i+1<s$, and put $\ell=\ell_{i+1}$.
By Lemma \ref{obs:lem7}, we have $\ux^{(i+1)}_\ell \in V_\ell(q_{i+1})$,
and we simply need to show that $\ux^{(i+1)}_\ell \in V_\ell(q_i)$.
If $\ell=n$, this is clear since $V_n(q_i)=\bR^n$.
Otherwise, the function $P_{\ell+1}$ is
constant on $[q_i,q_{i+1}]$.  So, Lemma \ref{obs:lem4} gives
$V_\ell(q_i)=V_\ell(q_{i+1})$ and the result follows.

\subsection*{Proof of Property 5).}
By Lemma \ref{obs:lem7},
the integer $|\det(\ux^{(i)}_1,\dots,\ux^{(i)}_n)| =
\|\ux^{(i)}_1\wedge\cdots\wedge\ux^{(i)}_n\|$ is positive
for each $i$. % with $0\le i<s$.
So, its logarithm is bounded below by $0$.  By Lemma
\ref{conv:lemma:L(n)} together with the estimates
\eqref{obs:eq:L-P}--\eqref{obs:eq:L-Pbis} and the
formula \eqref{conv:eq:P_1+...+P_n}, it is also
bounded above by
\[
 \log(n)-q_i+\sum_{j=1}^n L(\ux^{(i)}_j,q_i)
  \le \log(n)-q_i+\sum_{j=1}^n (P_j(q_i)+\epsilon)
   = n\epsilon+\log(n).
\]

\subsection*{Proof of Property 6).}
Fix an index $i$ with $1\le i<s$.  We set $k=k_i$ and define
\[
 \uy_j
  = \ux^{(i)}_1\wedge\cdots\wedge
      \widehat{\ux^{(i)}_j}\wedge\cdots\wedge\ux^{(i)}_n
 \quad (1\le j\le n).
\]
Since $\uy_k,\dots,\uy_n$ are $n-k+1$ linearly independent elements
of $\tbigwedge^{n-1}\bZ^n$, we have
\[
 L^{(n-1)}_{\uu,n-k+1}(t) \le \max\{L(\uy_k,t),\dots,L(\uy_n,t)\}
 \quad
 (t\ge 0).
\]
We use once again the fact that the functions $P_j$ with $j\neq k$
are constant on the interval $[q_i,q_i+\delta]$, while $P_k(t)=
P_k(q_i)+t-q_i$ for each $t\in[q_i,q_i+\delta]$.  In view of
\eqref{obs:eq:L-P} and \eqref{obs:eq:L-Pbis}, this implies that
\[
 |L(\ux^{(i)}_j,t)-P_j(t)|\le \epsilon
 \quad
 (1\le j\le n,\ q_i\le t\le q_i+\delta).
\]
From now on, we restrict to values of $t$ in the smaller interval
$[q_i,q_i+\delta/2]$
as this ensures that $P_j(t)\ge P_k(t)+\delta/2$ for $j=k+1,\dots,n$.
For each $j=1,\dots,n$, Lemma \ref{conv:lemma:L(k)} gives
\begin{align*}
 L(\uy_j,t)
  &\le \log(n)+L(\ux^{(i)}_1,t)+\cdots+
      \widehat{L(\ux^{(i)}_j,t)}+\cdots+L(\ux^{(i)}_n,t)\\
  &\le c_1+(n-1)\epsilon
       +P_1(t)+\cdots+\widehat{P_j(t)}+\cdots+P_n(t)\\
  &= c_1+(n-1)\epsilon+t-P_j(t)
  \qquad
  (\text{by \eqref{conv:eq:P_1+...+P_n}}),
\end{align*}
while Lemma \ref{conv:lemma:LLL} (iii) yields
\[
 L^{(n-1)}_{\uu,n-k+1}(t)
  \ge t-L_{\uu,k}(t)-2c_1
  \ge t-P_k(t)-2c_1-\epsilon.
\]
As $n\epsilon+3c_1<\delta/2$, this implies that
$L^{(n-1)}_{\uu,n-k+1}(t)>L(\uy_j,t)$ for $j=k+1,\dots,n$.
Therefore, we must have
$L^{(n-1)}_{\uu,n-k+1}(t) \le L(\uy_k,t)$, which in turn,
in view of the above estimates, implies that
\[
 t-P_k(t)-2c_1-\epsilon
 \le L(\uy_k,t)
 \le c_1+(n-1)\epsilon+t-P_k(t)
 \quad
 (q_i\le t\le q_i+\delta/2).
\]
As the interval $[q_i,q_i+\delta/2]$ has length greater
than $3c_1+n\epsilon$ and as the difference $t-P_k(t)$
is constant on this interval, the function $L(\uy_k,t)$,
having slope $0$ then $1$, must be constant on $[0,q_i]$
and so we obtain
\[
 \log\|\uy_k\| = L(\uy_k,q_i) \ge q_i-P_k(q_i)-2c_1-\epsilon.
\]
On the other hand, we have
\[
 \log\|\uy_k\|
 \le \sum_{j\neq k} \log\|\ux^{(i)}_j\|
 \le \sum_{j\neq k} (P_j(q_i)+\epsilon)
 = (n-1)\epsilon + q_i - P_k(q_i)
\]
using Property 3) and Formula \eqref{conv:eq:P_1+...+P_n}.
Then 6) follows.

%%%%%%%%%%%%%%%%%%%%%%%%%%%%%%%%%%%%%%%%%%%%%%%%%%%%%%
%
%  Distance and height of subspaces
%
%%%%%%%%%%%%%%%%%%%%%%%%%%%%%%%%%%%%%%%%%%%%%%%%%%%%%%

\section{Distance and height of subspaces}
\label{sec:dist}

We gather here definitions
and preliminary results that we will need in the next section
to deal with the inverse problem to Schmidt's
and Summerer's theory.
%Some of the constructions are inspired by
%\cite{Mo2012a, Ro2014}.

Let $n\ge 2$ be an integer. We say that a vector subspace $V$ of $\bR^n$
is \emph{defined over $\bQ$} if it spanned by elements of $\bQ^n$.  If
$V\neq 0$, this is equivalent to asking that $V\cap\bZ^n$ is a lattice
in $V$.  Then, following Schmidt in \cite{Sc1967}, we define the \emph{height}
of $V$, denoted $H(V)$, to be the co-volume of this lattice in $V$.
It is given by the formula
\[
 H(V) = \|\ux_1\wedge\cdots\wedge\ux_m\|
\]
where $(\ux_1,\dots,\ux_m)$ is any basis of $V\cap\bZ^n$, using the
natural Euclidean norm on $\tbigwedge^m\bR^n$ (see \S\ref{sec:conv}).
In particular, we note that $H(\bR^n)=1$.  We also set $H(0)=1$. The
next result is well known.

\begin{lemma}
 \label{dist:lem1}
Let $(\ux_1,\dots,\ux_n)$ be a basis of $\bZ^n$ and let $\uu$ be a
unit vector of $\bR^n$ perpendicular to $V:=\langle\ux_1,\dots,
\ux_{n-1}\rangle_\bR$.  Then, we have $H(V)=|\ux_n\cdot\uu|^{-1}$.
\end{lemma}

\begin{proof} We find
$1=\|\ux_1\wedge\cdots\wedge\ux_n\|
  =\|\ux_1\wedge\cdots\wedge\ux_{n-1}\|\,|\ux_n\cdot\uu|
  =H(V)|\ux_n\cdot\uu|$.
\end{proof}

We define the (projective) \emph{distance} between two non-zero points
$\ux$ and $\uy$ in $\bR^n$ by
\[
 \dist(\ux,\uy):=\frac{\|\ux\wedge\uy\|}{\|\ux\|\,\|\uy\|}.
\]
It represents the sine of the acute angle between the lines
spanned by $\ux$ and $\uy$ in $\bR^n$.
As a function on $(\bR^n\setminus\{0\})^2$, it is continuous and symmetric.
It also satisfies the triangle inequality:
\[
 \dist(\ux,\uz)\le \dist(\ux,\uy)+\dist(\uy,\uz)
 \quad
 (\ux,\uy,\uz\in\bR^n\setminus\{0\}).
\]

For any non-zero point $\ux\in\bR^n$ and any non-zero subspace $V$
of $\bR^n$, we define the distance from $\ux$ to $V$ by
\[
 \dist(\ux,V)
   :=\inf\{\dist(\ux,\uy)\,;\,\uy\in V\setminus\{0\}\}
    =\inf\{\dist(\ux,\uy)\,;\,\uy\in S^n\cap V\},
\]
where $S^n$ denotes the unit sphere of $\bR^n$.  Since $S^n\cap V$
is compact and since $\dist(\ux,\uy)$ is a continuous function of
$\uy\in\bR^n\setminus\{0\}$, this infimum is in fact a minimum,
achieved by at least one point $\uy$ in $S^n\cap V$.  We also define
$\dist(\ux,0)=1$ to be consistent with the following result.

\begin{lemma}
\label{dist:lem(x,V)}
Let $\ux\in\bR^n\setminus\{0\}$ and let $V$ be any
subspace of $\bR^n$.  Then, we have
\[
 \dist(\ux,V)=\frac{\|\proj_{V^\perp}(\ux)\|}{\|\ux\|}.
\]
Moreover, for any subspace $U$ of $\bR^n$ containing $V$, we have
$\dist(\ux,V)\ge \dist(\ux,U)$.
\end{lemma}

\begin{proof}
As the formula gives $\dist(\ux,0)=1$, we may assume that $V\neq 0$.
Since $\bR^n = V \perp V^\perp$, we have $\bigwedge^2\bR^n =
\bigwedge^2V \perp (V^\perp\wedge V) \perp \bigwedge^2V^\perp$.
In particular $\bigwedge^2V$ and $V^\perp\wedge V$ are orthogonal
subspaces of $\bR^n$.  Write $\ux=\uv+\uw$ with $\uv\in V$ and
$\uw\in V^\perp$ so that $\uw=\proj_{V^\perp}(\ux)$.  For any
$\uy\in V\setminus\{0\}$, we find
\[
 \dist(\ux,\uy)
   =\frac{\|(\uv\wedge\uy)+(\uw\wedge\uy)\|}{\|\ux\|\,\|\uy\|}
   \ge \frac{\|\uw\wedge\uy\|}{\|\ux\|\,\|\uy\|}
   = \frac{\|\uw\|}{\|\ux\|}
\]
with equality if and only if $\uv\wedge\uy=0$.  As there exists
$\uy\in V\setminus\{0\}$ satisfying the latter condition, this
proves the first assertion.  The second one is clear.
\end{proof}

Finally, given non-zero subspaces $V_1$ and $V_2$ of $\bR^n$,
we define the distance from $V_1$ to $V_2$ by
\[
 \dist(V_1,V_2)
   :=\sup\{\dist(\ux,V_2)\,;\,\ux\in V_1\setminus\{0\}\}
    =\sup\{\dist(\ux,V_2)\,;\,\ux\in S^n\cap V_1\},
\]
Since $S^n\cap V_1$ is compact and since the above lemma shows that
$\dist(\ux,V_2)$ is a continuous function of $\ux$ on
$\bR^n\setminus\{0\}$, this supremum is also achieved by some
point $\ux\in S^n\cap V_1$.  Note that this distance is not,
in general, a symmetric function of $V_1$ and $V_2$.  However,
it satisfies the following property.

\begin{lemma}
\label{dist:lem(x,V,V)}
Let $\ux\in\bR^n\setminus\{0\}$ and let $V_1,V_2$ be non-zero
subspaces of $\bR^n$.  Then, we have
\[
 \dist(\ux,V_2)\le\dist(\ux,V_1)+\dist(V_1,V_2).
\]
Moreover, if $V$ is also a non-zero subspace of $\bR^n$, then
\[
 \dist(V,V_2)\le\dist(V,V_1)+\dist(V_1,V_2).
\]
\end{lemma}

\begin{proof}
Choose $\uy\in S^n\cap V_1$ such that $\dist(\ux,V_1)=\dist(\ux,\uy)$
and $\uz\in S^n\cap V_2$ such that $\dist(\uy,V_2)=\dist(\uy,\uz)$.
Then
\[
 \begin{aligned}
  \dist(\ux,V_2)
   \le \dist(\ux,\uz)
   &\le \dist(\ux,\uy)+\dist(\uy,\uz)\\
   &= \dist(\ux,V_1)+\dist(\uy,V_2)
   \le \dist(\ux,V_1)+\dist(V_1,V_2).
 \end{aligned}
\]
The second inequality follows from this by choosing $\ux\in V$ such that
$\dist(V,V_2)=\dist(\ux,V_2)$ and then using $\dist(\ux,V_1)\le \dist(V,V_1)$.
\end{proof}

The next lemma shows a case where the distance between two vector
subspaces is a symmetric function of the subspaces.

\begin{lemma}
\label{dist:lem(V,V)}
Let $V_1$, $V_2$ be non-zero subspaces of co-dimension $1$ inside a
subspace $U$ of $\bR^n$. For $i=1,2$, choose a unit vector $\uu_i$ in
$U\cap V_i^\perp$.  Then, we have
\[
 \dist(V_1,V_2)=\dist(\uu_1,\uu_2).
\]
\end{lemma}

\begin{proof}
For the first assertion, we may assume that $V_1\neq V_2$ because
otherwise $\dist(V_1,V_2)=0=\dist(\uu_1,\uu_2)$.  Under this hypothesis,
we define $W:=V_1\cap V_2$ and, for $j=1,2$, we choose a unit vector
$\uv_j$ in $V_j\cap W^\perp$.  Then $(\uu_1,\uv_1)$ and $(\uu_2,\uv_2)$
are two orthonormal bases of $U\cap W^\perp$ and so
\[
 \dist(\uu_1,\uu_2)=\|\uu_1\wedge\uu_2\|=|\uv_1\cdot\uu_2|.
\]
Let $\ux$ be a non-zero element of $V_1$.  Upon writing $\ux=\uw+a\uv_1$
with $\uw\in W$ and $a\in\bR$, we find, using Lemma \ref{dist:lem(x,V)},
\[
 \dist(\ux,V_2)
 = \frac{\|\proj_{V_2^\perp}(\uw+a\uv_1)\|}{\|\uw+a\uv_1\|}
 = \frac{|a|\,|\uv_1\cdot\uu_2|}{\|\uw+a\uv_1\|}
 \le |\uv_1\cdot\uu_2|
\]
with equality if and only if $\uw=0$.  Thus, $\dist(V_1,V_2)
= |\uv_1\cdot\uu_2| = \dist(\uu_1,\uu_2)$.
\end{proof}

\begin{definition}
We say that a non-empty sequence $(\ux_1,\dots,\ux_m)$ of vectors of
$\bR^n$ is \emph{almost orthogonal} if it is linearly independent
and satisfies
\[
 \dist(\ux_j,\langle\ux_1,\dots,\ux_{j-1}\rangle_\bR)\ge 1-\frac{1}{2^{j-1}}
 \quad
 (2\le j \le m).
\]
\end{definition}

Thus, any sequence $(\ux)$ consisting of just one non-zero vector
$\ux\in\bR^n$ is almost orthogonal. By Lemma \ref{dist:lem(x,V)},
it follows that any non-empty subsequence of an almost orthogonal
sequence is almost orthogonal.  We conclude this section with two
more facts linked with this notion.

\begin{lemma}
 \label{dist:lemH}
Let $(\ux_1,\dots,\ux_m)$ be an almost orthogonal primitive
$m$-tuple of points of $\bZ^n$ and let
$U:=\langle\ux_1,\dots,\ux_m\rangle_\bR$.  Then, we have
\[
 e^{-2}\|\ux_1\|\cdots\|\ux_m\|
 \le H(U)
 \le \|\ux_1\|\cdots\|\ux_m\|.
\]
\end{lemma}

\begin{proof}
Since $H(U)=\|\ux_1\wedge\cdots\wedge\ux_m\|$, the upper
bound is clear.  For the lower bound, we use induction
to show that
\begin{equation}
 \label{dist:eq:prod}
 \|\ux_1\wedge\cdots\wedge\ux_k\|
 \ge \|\ux_1\|\cdots\|\ux_k\|
     \prod_{j=2}^k (1-2^{1-j})
 \quad
 (1\le k\le m).
\end{equation}
Since $\prod_{j=2}^m (1-2^{1-j}) \ge \prod_{j=2}^\infty \exp(-2^{2-j})
= e^{-2}$, this will complete the proof.  For $k=1$, the inequality
\eqref{dist:eq:prod} is clear.  Suppose that it holds for some
integer $k$ with $1\le k<m$.  Set $V=\langle \ux_1,\dots,\ux_k\rangle_\bR$.
Using Lemma \ref{dist:lem(x,V)}, we find
\begin{align*}
 \|\ux_1\wedge\cdots\wedge\ux_{k+1}\|
 &= \|\ux_1\wedge\cdots\wedge\ux_k\|\,\|\proj_{V^\perp}(\ux_{k+1})\| \\
 &= \|\ux_1\wedge\cdots\wedge\ux_k\|\,\|\ux_{k+1}\|\, \dist(\ux_{k+1},V)\\
 &= \|\ux_1\wedge\cdots\wedge\ux_k\|\,\|\ux_{k+1}\|\, (1-2^{-k}).
\end{align*}
So, \eqref{dist:eq:prod} also holds with $k$ replaced by $k+1$.
\end{proof}

\begin{lemma}
\label{dist:lemVVH}
Let $k,\ell,m\in\{1,\dots,n\}$ with $k<\ell\le m$,
let $U$ be a subspace of\/ $\bR^n$ of dimension $m$ defined over
$\bQ$, and let $(\uy_1,\dots,\uy_m)$ be a basis of $U\cap\bZ^n$.
Define
\[
 V_1=\big\langle \uy_1,\dots,\widehat{\uy_\ell\,},\dots,\uy_m\big\rangle_\bR
 \et
 V_2=\big\langle \uy_1,\dots,\widehat{\uy_k},\dots,\uy_m\big\rangle_\bR.
\]
Then, we have
\begin{equation}
 \label{dist:lemVVH:eq1}
 \dist(V_1,V_2)=\frac{H(V_1\cap V_2)H(U)}{H(V_1)H(V_2)}.
\end{equation}
Moreover, if $(\uy_1,\dots,\widehat{\uy_\ell\,},\dots,\uy_m)$ and
$(\uy_1,\dots,\widehat{\uy_k},\dots,\uy_m)$ are almost orthogonal,
then
\begin{equation}
 \label{dist:lemVVH:eq2}
 \dist(V_1,V_2)\le \frac{e^4H(U)}{\|\uy_1\|\cdots\|\uy_m\|}.
\end{equation}
\end{lemma}

\begin{proof}
Set $W = V_1\cap V_2 = \langle \uy_1\wedge\cdots\wedge\widehat{\uy_k}
\wedge\cdots\wedge\widehat{\uy_\ell\,}\wedge\cdots\wedge\uy_m
\rangle_\bR$.  For $j=1,2$, we choose unit vectors
$\uu_j$ in $U\cap V_j^\perp$ and $\uv_j$ in $V_j\cap W^\perp$.
Then $(\uu_1,\uv_1)$ and $(\uu_2,\uv_2)$ are orthonormal bases
of $U\cap W^\perp$.  So, $\uv_1\wedge\uv_2 = \pm \uu_1\wedge\uu_2$
and, by Lemma \ref{dist:lem(V,V)}, we obtain $\|\uv_1\wedge\uv_2\|
=\dist(V_1,V_2)$.

Define $\omega=
\uy_1\wedge\cdots\wedge\widehat{\uy_k}
     \wedge\cdots\wedge\widehat{\uy_\ell\,}\wedge\cdots\wedge\uy_m$,
so that $H(W)=\|\omega\|$.  Upon writing $\uy_k=\uw_1+a_1\uv_1$ and
$\uy_\ell=\uw_2+a_2\uv_2$ with $\uw_1,\uw_2\in W$ and
$a_1,a_2\in\bR$, we find
\begin{align*}
 H(V_j)
  &= \|\omega\wedge(\uw_j+a_j\uv_j)\|
  = \|\omega\|\,|a_j|
  \quad (j=1,2), \\
 H(U)
  &= \|\omega\wedge(\uw_1+a_1\uv_1)\wedge(\uw_2+a_2\uv_2)\|\\
  &= \|\omega\|\,\|a_1a_2\uv_1\wedge\uv_2\|
  = \|\omega\|\,|a_1a_2|\, \dist(V_1,V_2),
\end{align*}
and \eqref{dist:lemVVH:eq1} follows because $a_1a_2\neq 0$.

If $(\ux_1,\dots,\widehat{\ux_\ell\,},\dots,\ux_m)$ and
$(\uy_1,\dots,\widehat{\uy_k},\dots,\uy_m)$ are almost
orthogonal, Lemma \ref{dist:lemH} gives
\[
 H(V_1)\ge \frac{\|\uy_1\|\dots\|\uy_m\|}{e^2 \|\uy_\ell\|}\,
 \et
 H(V_2)\ge \frac{\|\uy_1\|\dots\|\uy_m\|}{e^2 \|\uy_k\|}\,.
\]
Then \eqref{dist:lemVVH:eq2} follows from \eqref{dist:lemVVH:eq1}
using $H(W)=\|\omega\|\le (\|\uy_1\|\dots\|\uy_m\|)/(\|\uy_k\|\,\|\uy_\ell\|)$.
\end{proof}

%\newpage

%%%%%%%%%%%%%%%%%%%%%%%%%%%%%%%%%%%%%%%%%%%%%%%%%%%%%%
%
%  The inverse problem
%
%%%%%%%%%%%%%%%%%%%%%%%%%%%%%%%%%%%%%%%%%%%%%%%%%%%%%%

\section{The inverse problem for rigid systems}
\label{sec:inv}

The goal of this section is to prove a partial converse
to Schmidt's and Summerer's theorem \ref{conv:thm:SS} for
rigid systems with large mesh.  Motivated by Theorem \ref{obs:thm},
we start by constructing recursively a sequence of bases of
$\bZ^n$ with several strong properties.  The key to the recurrence
is provided by the following lemma.

\begin{lemma}
\label{inv:lem1}
Let $h,k,\ell\in\{1,\dots,n\}$ with $h\le \ell$ and
$k<\ell$, let $(\ux_1,\dots,\ux_n)$ be a basis
of $\bZ^n$, and let $A$ be a real number with
$A\ge 2^{\ell} (\|\ux_1\|+\cdots+\|\ux_\ell\|)$.
Then, there exists a basis $(\uy_1,\dots,\uy_n)$ of $\bZ^n$
satisfying
\begin{itemize}
 \item[1)]
   $(\uy_1,\dots,\widehat{\uy_\ell},\dots,\uy_n)
     =(\ux_1,\dots,\widehat{\ux_h},\dots,\ux_n)$,
 \item[2)]
   $\uy_\ell \in
    \ux_h
     +\big\langle
      \ux_1,\dots,\widehat{\ux_h},\dots,\ux_\ell
      \big\rangle_\bZ$,
 \item[3)]
    $A\le \|\uy_\ell\|\le 2A$,
 \item[4)]
    $\disp \dist(\uy_\ell,
       \langle
        \uy_1,\dots,\widehat{\uy_k},\dots,\uy_{\ell-1}
       \rangle_\bR)
     \ge 1-\frac{1}{2^{\ell-1}}$.
\end{itemize}
Moreover, if $(\ux_1, \dots, \widehat{\ux_h}, \dots, \ux_\ell)$
is almost orthogonal, then $(\uy_1, \dots, \widehat{\uy_k},
\dots,\uy_\ell)$ is almost orthogonal.
\end{lemma}

Note that the condition 4) is stronger than needed in order to
establish that the sequence $(\uy_1,\dots,\widehat{\uy_k},\dots,\uy_\ell)$
is almost orthogonal.  In view of the definition, a lower bound of
${1-1/2^{\ell-2}}$ would suffice instead of ${1-1/2^{\ell-1}}$.
However, this stronger requirement will show its importance later.
Note also that 4) trivially holds when $\ell=2$.

\begin{proof}
We use 1) as a definition of the vectors $\uy_1,\dots,\widehat{\uy_\ell},
\dots,\uy_n$.  Then, $(\uy_1,\dots,\uy_n)$ is a basis of $\bZ^n$
for any choice of $\uy_\ell$ in
$\ux_h+\langle\ux_1,\dots,\widehat{\ux_h},\dots,\ux_n\rangle_\bZ$.
A fortiori, it is a basis of $\bZ^n$ for any choice of $\uy_\ell$
satisfying 2).  To prove the existence of a point $\uy_\ell$
satisfying 2), 3) and 4), we first observe that the hypothesis
$h\le \ell$ yields
\[
 (\uy_1,\dots,\uy_{\ell-1})
   = (\ux_1,\dots,\widehat{\ux_h},\dots,\ux_\ell).
\]
Consider the chain of subspaces $W\subset V\subset U$ of relative
codimension $1$ given by
\begin{align*}
 W&=\langle\uy_1,\dots,\widehat{\uy_k},\dots,\uy_{\ell-1}\rangle_\bR,\\
 V&=\langle\uy_1,\dots,\uy_{\ell-1}\rangle_\bR
  =\langle\ux_1,\dots,\widehat{\ux_h},\dots,\ux_\ell\rangle_\bR,\\
 U&=\langle\ux_1,\dots,\ux_{\ell}\rangle_\bR.
\end{align*}
Choose unit vectors $\uu\in U\cap V^\perp$ and $\uv\in V\cap W^\perp$.
Then $\{\uu,\uv\}$ is an orthonormal basis of $U\cap W^\perp$ and
we have $\proj_{V^\perp}(\ux_h) = c\uu$ where
$c = \ux_h\cdot\uu$.  Put $B=(3/2)A$.  Since the vector
$c\uu+B\uv-\ux_h$ belongs to $V$, we can write
\[
 c\uu+B\uv = \ux_h+\sum_{j=1}^{\ell-1} c_j \uy_j
\]
for some $c_1,\dots,c_{\ell-1}\in\bR$.  We choose
\[
 \uy_\ell
   := \ux_h+\sum_{j=1}^{\ell-1} \lceil c_j\rceil \uy_j
   = c\uu+B\uv+\sum_{j=1}^{\ell-1} \epsilon_j \uy_j
\]
where $\epsilon_j=\lceil c_j\rceil-c_j\in[0,1)$ for $j=1,\dots,\ell-1$. Then,
the condition 2) is fulfilled and
\[
 \big| \|\uy_\ell\| - B \big|
   \le \|\uy_\ell-B\uv\|
   \le |c| + \sum_{j=1}^{\ell-1} \|\uy_j\|
   \le \sum_{j=1}^{\ell} \|\ux_j\|
   \le \frac{B}{2^{\ell}},
\]
so that $(1-2^{-\ell})B\le \|\uy_\ell\|\le (1+2^{-\ell})B$.
Since $\ell\ge 2$, this shows that 3) holds.
Moreover, we find that
\[
 \|\proj_{W^\perp}(\uy_\ell)\|
  \ge |\uy_\ell\cdot\uv|
   = |B+\epsilon_k\uy_k\cdot\uv|
  \ge B-\|\uy_k\|
  \ge(1-2^{-\ell})B
\]
and therefore
\[
 \dist(\uy_\ell,W)
   = \frac{\|\proj_{W^\perp}(\uy_\ell)\|}{\|\uy_\ell\|}
   \ge \frac{1-2^{-\ell}}{1+2^{-\ell}}
   \ge 1-\frac{1}{2^{\ell-1}}.
\]
Thus $\uy_\ell$ also satisfies 4).  Finally, if
$(\ux_1, \dots, \widehat{\ux_h}, \dots, \ux_\ell)
=(\uy_1,\dots,\uy_{\ell-1})$ is almost orthogonal,
then its subsequence
$(\uy_1, \dots, \widehat{\uy_k}, \dots,\uy_{\ell-1})$
is almost orthogonal, and by 4) we conclude that
$(\uy_1, \dots, \widehat{\uy_k}, \dots,\uy_\ell)$
is almost orthogonal.
\end{proof}

%%%%%

The next lemma provides us with an initial
basis of $\bZ^n$ to start the recurrence.

\begin{lemma}
\label{inv:lem3}
Let $k$ be an integer with $1\le k\le n$ and let
$A_1,\dots,A_n$ be real numbers satisfying $A_1\ge 8$
and $A_j\ge 2^{j+3}A_{j-1}$ for $j=2,\dots,n$.
Then there exists a basis $(\uy_1,\dots,\uy_n)$ of
$\bZ^n$ which satisfies the following conditions
\begin{itemize}
 \item[(i)] $(\uy_1,\dots,\uy_{n-1})$
    and $(\uy_1,\dots,\widehat{\uy_k},\dots,\uy_n)$
    are almost orthogonal,
 \item[(ii)] $A_j\le\|\uy_j\|\le 2A_j$ for $j=1,\dots,n$,
 \item[(iii)]
  $\dist(\uy_n,
   \langle
   \uy_1,\dots,\widehat{\uy_k},\dots,\uy_{n-1}
   \rangle_\bR)
   \ge 1-1/2^{n-1}$ \text{ if \ $k<n$.}
\end{itemize}
\end{lemma}

Note that, like in the preceding lemma, the lower bound
$1-1/2^{n-1}$ imposed in condition (iii) is larger than the
lower bound $1-1/2^{n-2}$ which follows from condition (i)
if $k<n$.

\begin{proof}
Let $(\ue_1,\dots,\ue_n)$ denote the canonical basis of
$\bZ^n$.  Lemma \ref{inv:lem1} allows us to construct
recursively a sequence of integer points
$\ux_1,\dots,\ux_n\in\bZ^n$ starting with $\ux_1=\ue_1$
which, for each $j=2,\dots,n$, satisfy the following
properties:
\begin{itemize}
 \item[1)] $(\ux_1,\dots,\ux_{j-1},\ue_j,\dots,\ue_n)$
   is a basis of $\bZ^n$,
 \item[2)] $\ux_j\in\ue_j+\langle\ux_1,\dots,\ux_{j-1}\rangle_\bZ$,
 \item[3)] $A_{j-1}\le \|\ux_j\|\le 2A_{j-1}$,
 \item[4)] $\dist(\ux_j,\langle\ux_2,\dots,\ux_{j-1}\rangle_\bR)
   \ge 1-1/2^{j-1}$ \ if \ $j\ge 3$.
\end{itemize}
Indeed, the condition 1) holds for $j=2$ because $\ux_1=\ue_1$.
Suppose that we have constructed $\ux_1,\dots,\ux_{\ell-1}$
for some $\ell$ with $2\le \ell\le n$ so that 1) holds when
$2\le j\le\ell$ while 2), 3) and 4) hold when $2\le j\le \ell-1$.
We apply Lemma \ref{inv:lem1} to the basis
$(\ux_1,\dots,\ux_{\ell-1},\ue_\ell,\dots,\ue_n)$ of $\bZ^n$
with the choice of $h=\ell$, $k=1$ and $A=A_{\ell-1}$.  Since
\[
 2^\ell(\|\ux_1\|+\cdots+\|\ux_{\ell-1}\|+\|\ue_\ell\|)
   \le 2^{\ell+1}\Big(1+\sum_{2\le j\le \ell-1} A_{j-1}\Big)
   \le A_{\ell-1},
\]
it provides a new point $\ux_\ell$ of $\bZ^n$
which satisfies the condition 1) for $j=\ell+1$ as well as
the conditions 2), 3) and 4) for $j=\ell$.  This proves our
claim and, in the case $\ell=n$, shows that $(\ux_1,\dots,\ux_n)$
is a basis of $\bZ^n$.  Moreover, since 4) holds for $j=3,\dots,n$,
we also note that $(\ux_2,\dots,\ux_n)$ is almost orthogonal.

If $k<n$, we apply once again Lemma \ref{inv:lem1}, this time to the
basis $(\ux_1,\dots,\ux_n)$ with the choice of $h=1$, $\ell=n$ and the
given value of $k$.  Since
\[
 2^n(\|\ux_1\|+\cdots+\|\ux_n\|)
   \le 2^{n+1}\Big(1+\sum_{2\le j\le n} A_{j-1}\Big)
   \le A_n,
\]
it provides a new basis of $\bZ^n$ of the form $(\uy_1,\dots,\uy_n)
=(\ux_2,\dots,\ux_n,\uy_n)$ with all the requested properties.
Otherwise, it suffices to choose
$(\uy_1,\dots,\uy_n)=(\ux_2,\dots,\ux_n,\ux_1+B\ux_2)$
with $B=\left\lceil A_n/\|\ux_2\|\,\right\rceil+1$.
\end{proof}

From now on, we fix a choice of $s\in\bN^*\cup\{\infty\}$
and put
\[
 C=2^{n+3}e^4.
\]
We suppose that, for each integer $i$ with
$0\le i<s$, we are given a point $\uA^{(i)}=(A_1^{(i)},\dots,A_n^{(i)})$
in $\bR^n$ and integers $k_i$ and $\ell_i$ satisfying the conditions
\begin{align}
 &1\le k\le \ell_0=n \et
 1\le k_i< \ell_i\le n \quad \text{if \ $i\ge 1$,}
 \label{inv:H1}\\
 &A_1^{(i)}\ge C, \quad
 A_j^{(i)}\ge A_{j-1}^{(i)}C \quad \text{for \ $j=2,\dots,n$,}
 \label{inv:H2}\\
 &k_{i-1}\le \ell_i \et
 A_{\ell_i}^{(i)} \ge A_{\ell_i}^{(i-1)}C \quad \text{if \ $i\ge 1$,}
 \label{inv:H3}\\
 &(A_1^{(i)},\dots,\widehat{A_{\ell_i}^{(i)}},\dots,A_n^{(i)})
  =(A_1^{(i-1)},\dots,\widehat{A_{k_{i-1}}^{(i-1)}},\dots,A_n^{(i-1)})
  \quad
  \text{if \ $i\ge 1$.}
 \label{inv:H4}
\end{align}
We now combine the previous lemmas to establish the following result.

\begin{proposition}
 \label{inv:prop1}
For each integer $i$ with $0\le i< s$, there exists
a basis $(\ux_1^{(i)},\dots,\ux^{(i)}_n)$ of $\bZ^n$
which satisfies the following properties:
\begin{itemize}
 \item[1)]
  $(\ux_1^{(0)},\dots,\ux^{(0)}_{n-1})$
  is almost orthogonal,
 \item[2)]
  $(\ux_1^{(i)},\dots,\widehat{\ux^{(i)}_{k_i}},\dots,\ux^{(i)}_n)$
  is almost orthogonal,
 \item[3)]
  $A^{(i)}_j \le \|\ux^{(i)}_j\| \le 2A^{(i)}_j$
  for $j=1,\dots,n$,
 \item[4)] $\disp
  \dist\Big(\ux^{(i)}_{\ell_i},
   \big\langle
   \ux_1^{(i)},\dots,\widehat{\ux^{(i)}_{k_i}},\dots,\ux^{(i)}_{\ell_i-1}
   \big\rangle_\bR\Big)
   \ge 1-\frac{1}{2^{\ell_i-1}}$ \ \text{if \ $k_i<\ell_i$,}
 \item[5)]
   $\ux^{(i)}_{\ell_{i}}
    \in
    \ux^{(i-1)}_{k_{i-1}}
     + \big\langle
       \ux^{(i-1)}_1,\dots,\widehat{\ux^{(i-1)}_{k_{i-1}}},\dots,\ux^{(i-1)}_{\ell_i}
       \big\rangle_\bZ$
    \ \text{if \ $i\ge 1$,}
 \item[6)]
  $(\ux_1^{(i)},\dots,\widehat{\ux_{\ell_{i}}^{(i)}},\dots,\ux^{(i)}_n)
   = (\ux_1^{(i-1)},\dots,\widehat{\ux_{k_{i-1}}^{(i-1)}},\dots,\ux^{(i-1)}_n)$
   \ \text{if \ $i\ge 1$.}
\end{itemize}
Moreover, if $\uu_i$ denotes a unit vector orthogonal to
$\langle
 \ux_1^{(i)},\dots,\widehat{\ux^{(i)}_{k_i}},\dots,\ux^{(i)}_n
 \rangle_\bR$
then, for each pair of integers $i$ and $j$ with $0\le i<j< s$, we have
\[
 \dist(\uu_i,\uu_j) \le \frac{2e^4}{\|\ux_1^{(i+1)}\|\cdots\|\ux_n^{(i+1)}\|}.
\]
\end{proposition}

\begin{proof}
We prove the first assertion by induction on $i$.  To construct
% the initial basis
$(\ux_1^{(0)},\dots,\ux_n^{(0)})$, we apply
Lemma \ref{inv:lem3} with the choice of $k=k_0$ and $A_j=A_j^{(0)}$
for $j=1,\dots,n$.  Thanks to \eqref{inv:H2}, the hypotheses of the
lemma are satisfied and, since $\ell_0=n$, the basis that we obtain
satisfies the conditions 1) to 4) of the proposition for $i=0$.
The conditions 5) and 6) are vacuous.

Now, let $t$ be an integer with $1\le t<s$.  Suppose that,
for $i=0,\dots,t-1$, we have constructed a basis
$(\ux_1^{(i)},\dots,\ux_n^{(i)})$ of $\bZ^n$
which satisfies all conditions 1) to 6).
To construct the next basis, we apply Lemma \ref{inv:lem1} to
$(\ux_1^{(t-1)},\dots,\ux_n^{(t-1)})$ for the choice of
$h=k_{t-1}$, $k=k_t$, $\ell=\ell_t$ and $A=A_{\ell_t}^{(t)}$.
Since 3) holds for $i=t-1$, we find with the help of \eqref{inv:H2}
and \eqref{inv:H3} that
\[
 2^{\ell_t}(\|\ux_1^{(t-1)}\|+\cdots+\|\ux_{\ell_t}^{(t-1)}\|)
  \le 2^{\ell_t+1}(A_1^{(t-1)}+\cdots+A_{\ell_t}^{(t-1)})
  \le 2^{\ell_t+2}A_{\ell_t}^{(t-1)}
  \le A_{\ell_t}^{(t)}.
\]
So, the lemma produces a new basis $(\ux_1^{(t)},\dots,\ux_n^{(t)})$
of $\bZ^n$ which satisfies the conditions 4) to 6) for $i=t$, as well as
$A_{\ell_t}^{(t)}\le \|\ux_{\ell_t}^{(t)}\| \le 2A_{\ell_t}^{(t)}$.
Then, combining our hypothesis that 3) holds for $i=t-1$ with the
equalities 6) and \eqref{inv:H4} for $i=t$, we conclude that 3)
holds for $i=t$.

To complete the inductive step, it remains to show that
$(\ux_1^{(t)},\dots,\widehat{\ux_{k_{t}}^{(t)}},\dots,\ux_n^{(t)})$
is almost orthogonal.  This is the most delicate part of the argument.
To prove this, we use induction on $m$ to show that
$(\ux_1^{(t)},\dots,\widehat{\ux_{k_{t}}^{(t)}},\dots,\ux_m^{(t)})$
is almost orthogonal for $m=\ell_t,\dots,n$.  When $m=\ell_t$, this
follows from our construction, thanks to the last assertion of
Lemma \ref{inv:lem1}, because $(\ux_1^{(t-1)},\dots,
\widehat{\ux_{k_{t-1}}^{(t-1)}},\dots,\ux_{\ell_t}^{(t-1)})$
is almost orthogonal.  If $\ell_t=n$, we are done.  Otherwise, fix an
integer $m$ with $\ell_t\le m <n$ and assume that our claim is
true for that value of $m$.  Since $\ell_0=n >m$, there exists an
index $r$ with $0\le r <t$ such that
\begin{equation}
\label{inv:prop:eq1}
 \ell_r>m \et \ell_{r+1},\dots,\ell_t\le m.
\end{equation}
In particular, we have $k_r,\dots,k_t\le m$.  Thus
$(\ux_1^{(i)},\dots,\widehat{\ux_{k_{i}}^{(i)}},\dots,\ux_m^{(i)})$
is almost orthogonal for $i=r,\dots,t$, while
\[
  (\ux_1^{(i)},\dots,\widehat{\ux_{\ell_{i}}^{(i)}},\dots,\ux^{(i)}_m)
   =
  (\ux_1^{(i-1)},\dots,\widehat{\ux_{k_{i-1}}^{(i-1)}},\dots,\ux^{(i-1)}_m)
\]
is also almost orthogonal for $i=r+1,\dots,t$.  Define
\[
 U^{(i)}=\big\langle \ux_1^{(i)},\dots,\ux^{(i)}_m \big\rangle_\bR
 \et
 V^{(i)}=
   \big\langle
   \ux_1^{(i)},\dots,\widehat{\ux_{k_{i}}^{(i)}},\dots,\ux^{(i)}_m
   \big\rangle_\bR
 \quad
 (r\le i\le t).
\]
As $V^{(i-1)} = \langle \ux_1^{(i)},\dots,
\widehat{\ux_{\ell_{i}}^{(i)}},\dots,\ux^{(i)}_m \rangle_\bR$
for $i=r+1,\dots,t$, Lemma \ref{dist:lemVVH} gives
\[
 \dist(V^{(i-1)},V^{(i)})
  \le
  e^4\frac{H(U^{(i)})}{\|\ux_1^{(i)}\|\cdots\|\ux_m^{(i)}\|}
  \quad
  (r+1\le i\le t).
\]
By 5), we also have
$U^{(i-1)}
  = V^{(i-1)}+\langle\ux_{k_{i-1}}^{(i-1)}\rangle_\bR
  = V^{(i-1)}+\langle\ux_{\ell_{i}}^{(i)}\rangle_\bR
  = U^{(i)}
$
for $i=r+1,\dots,t$.  Thus $U^{(r)}=U^{(r+1)}=\cdots=U^{(t)}$ and therefore,
using Lemma \ref{dist:lem(x,V,V)}, we obtain
\[
 \dist(V^{(r)},V^{(t)})
  \le \sum_{i=r+1}^t \dist(V^{(i-1)},V^{(i)})
  \le e^4 \sum_{i=r+1}^t
  \frac{H(U^{(r)})}{\|\ux_1^{(i)}\|\cdots\|\ux_m^{(i)}\|}.
\]
With the help of 6), 3), \eqref{inv:H2} and \eqref{inv:H3},
we also note that
\begin{equation}
 \label{inv:prop:eq3}
 \frac{\|\ux_1^{(i-1)}\|\cdots\|\ux_m^{(i-1)}\|}
      {\|\ux_1^{(i)}\|\cdots\|\ux_m^{(i)}\|}
 =
 \frac{\|\ux_{k_{i-1}}^{(i-1)}\|}{\|\ux_{\ell_i}^{(i)}\|}
 \le
 \frac{2A_{k_{i-1}}^{(i-1)}}{A_{\ell_i}^{(i)}}
 \le
 \frac{2A_{\ell_i}^{(i-1)}}{A_{\ell_i}^{(i)}}
 \le
 \frac{2}{C}
\end{equation}
for $i=r+1,\dots,t$. This yields
\[
 \dist(V^{(r)},V^{(t)})
  \le 2e^4
  \frac{H(U^{(r)})}{\|\ux_1^{(r+1)}\|\cdots\|\ux_m^{(r+1)}\|}
  \le 2e^4
  \frac{\|\ux_1^{(r)}\|\cdots\|\ux_m^{(r)}\|}
       {\|\ux_1^{(r+1)}\|\cdots\|\ux_m^{(r+1)}\|}
  \le \frac{4e^4}{C} \le \frac{1}{2^n}.
\]
According to 6) and \eqref{inv:prop:eq1}, we also have $\ux^{(r)}_{m+1} =
\ux^{(r+1)}_{m+1} = \cdots = \ux^{(t)}_{m+1}$.  Therefore, we
conclude, by Lemma \ref{dist:lem(x,V,V)}, that
\[
 \dist(\ux^{(t)}_{m+1},V^{(t)})
 \ge
 \dist(\ux^{(t)}_{m+1},V^{(r)})-\dist(V^{(r)},V^{(t)})
 \ge
 \dist(\ux^{(r)}_{m+1},V^{(r)})-\frac{1}{2^n}.
\]
We recall at this point that $\ell_r>m$.  If $r>0$ and
$\ell_r>m+1$, then $(\ux_1^{(r)},\dots,\ux_{m+1}^{(r)})$ is
almost orthogonal because it is a subsequence of
\[
 (\ux_1^{(r)},\dots,\widehat{\ux_{\ell_{r}}^{(r)}},\dots,\ux^{(r)}_n)
 =
 (\ux_1^{(r-1)},\dots,\widehat{\ux_{k_{r-1}}^{(r-1)}},\dots,\ux^{(r-1)}_n)
\]
which is almost orthogonal.  The same is true if $r=0$ and
$n=\ell_0>m+1$ because, in that case, it is a subsequence of
$(\ux_1^{(0)},\dots,\ux_{n-1}^{(0)})$ which, by condition 1),
is almost orthogonal.  So, if $\ell_r>m+1$, we obtain
\[
 \dist(\ux^{(r)}_{m+1},V^{(r)})
 \ge \dist(\ux^{(r)}_{m+1},U^{(r)})
  =  \dist(\ux^{(r)}_{m+1},\langle\ux_1^{(r)},\dots,\ux_m^{(r)}\rangle_\bR )
 \ge 1-\frac{1}{2^m}.
\]
If $\ell_r=m+1$, the inequality $\dist(\ux^{(r)}_{m+1},V^{(r)})
\ge 1-1/2^m$ also holds as a direct consequence of 4) for $i=r$ because
$k_r\le m<\ell_r$ (in fact, this is precisely the reason why
we need this condition 4)).  So, in all cases, we conclude that
\[
 \dist(\ux^{(t)}_{m+1},V^{(t)})
 \ge
 \dist(\ux^{(r)}_{m+1},V^{(r)})-\frac{1}{2^n}
 \ge
 1-\frac{1}{2^m}-\frac{1}{2^n}
 \ge
 1-\frac{1}{2^{m-1}}.
\]
Thus $(\ux_1^{(t)},\dots,\widehat{\ux_{k_{t}}^{(t)}},\dots,\ux_{m+1}^{(t)})$
is almost orthogonal, and our claim follows by the induction principle.

Having proved the first part of the proposition, we know that
$(\ux_1^{(i)},\dots,\widehat{\ux_{k_{i}}^{(i)}},\dots,\ux^{(i)}_n)$
is almost orthogonal for each $i$ with $0\le i<s$.  Define
\[
 V^{(i)}=
   \langle
   \ux_1^{(i)},\dots,\widehat{\ux_{k_{i}}^{(i)}},\dots,\ux^{(i)}_n
   \rangle_\bR
 \quad
 (0\le i<s).
\]
When $i\ge 1$, the condition 6) gives $V^{(i-1)} = \langle
\ux_1^{(i)},\dots,\widehat{\ux_{\ell_{i}}^{(i)}},\dots,\ux^{(i)}_n
\rangle_\bR$ and so $V^{(i-1)}+V^{(i)}
  = \langle \ux_1^{(i)},\dots,\ux^{(i)}_n \rangle_\bR
  =\bR^n$.
Since $H(\bR^n)=1$, we deduce from Lemmas \ref{dist:lem(V,V)}
and \ref{dist:lemVVH} that
\[
 \dist(\uu_{i-1},\uu_i)
 =
 \dist(V^{(i-1)},V^{(i)})
  \le
  \frac{e^4}{\|\ux_1^{(i)}\|\cdots\|\ux_n^{(i)}\|}.
\]
We also note that
\eqref{inv:prop:eq3} holds with $m=n$ for any integer $i$ with
$1\le i<s$.  So, for integers $r,t$ with
$0\le r<t<s$, we conclude that
\[
 \dist(\uu_r,\uu_t)
 \le
 \sum_{i=r+1}^t\dist(\uu_{i-1},\uu_i)
  \le
 \sum_{i=r+1}^t
 \frac{e^4}{\|\ux_1^{(i)}\|\cdots\|\ux_n^{(i)}\|}
  \le
 \frac{2e^4}{\|\ux_1^{(r+1)}\|\cdots\|\ux_n^{(r+1)}\|}.
 \qedhere
\]
\end{proof}

The next result completes the above proposition by constructing
a unit vector $\uu$ and by estimating the distance function of
the points $\ux^{(i)}_1,\dots,\ux^{(i)}_n$ with respect to the
convex body $\cC_\uu(Q)$ in appropriate ranges for $Q$.

\begin{proposition}
\label{inv:prop2}
Let the notation be as in the previous proposition.  Define
\[
 Q_i=A_1^{(i)}\cdots A_n^{(i)}
 \quad
 (0\le i <s),
\]
and set $Q_s=\infty$ if $s\neq \infty$.
Then there exists a unit vector
$\uu\in\bR^n$ with the property that, for each $i$ with
$0\le i<s$ and each $Q\in[Q_i,Q_{i+1})$, we have
\begin{itemize}
 \item[1)]
   $A_j^{(i)}
   \le \lambda(\ux_j^{(i)},\cC_\uu(Q))
   \le 8e^4A_j^{(i)}$
   \quad \text{for any $j\in\{1,\dots,n\}$ with $j\neq k_i$,}\\[-3pt]
 \item[2)]
   $\disp\frac{A_{k_i}^{(i)}Q}{2^{n}Q_i}
   \le \lambda(\ux_{k_i}^{(i)},\cC_\uu(Q))
   \le \frac{8A_{k_i}^{(i)}Q}{Q_i}$.
\end{itemize}
\end{proposition}

\begin{proof}
We will use freely the inequalities $A_j^{(i)}\le
\|\ux_j^{(i)}\|\le 2A_j^{(i)}$ $(1\le j\le n,\ 0\le i<s)$
coming from the condition 3) of Proposition \ref{inv:prop1},
as well as the inequalities \eqref{inv:H2} to \eqref{inv:H4}
involving the numbers $A_j^{(i)}$.
The construction of the point $\uu$ is based on the second
assertion of Proposition \ref{inv:prop1}.  When $s=\infty$,
this statement implies that the image of the sequence
$(\uu_i)_{i\ge 1}$ in $\bP^{n-1}(\bR)$ converges to the
class of a unit vector $\uu\in\bR^n$ such that
\[
 \dist(\uu_i,\uu)
 \le \frac{2e^4}{\|\ux_1^{(i+1)}\|\cdots\|\ux_n^{(i+1)}\|}
 \quad (0\le i<s).
\]
When $s\neq \infty$, these inequalities remain true for the
choice of $\uu=\uu_{s-1}$, provided that we interpret the right
end side as $0$ when $i=s-1$.  We now replace each $\uu_i$
by $\pm\uu_i$ so that $\uu_i\cdot\uu\ge 0$.  Then, we obtain
\[
 \|\uu_i-\uu\|
 \le 2\dist(\uu_i,\uu)
 \le \frac{4e^4}{Q_{i+1}}
 \quad (0\le i<s).
\]
For the rest of the argument, we fix an integer $i$ with $0\le i<s$
and a number $Q\in [Q_i,Q_{i+1})$.  Our goal is to estimate
\begin{equation}
 \label{inv:prop2:eq1}
 \lambda(\ux_j^{(i)},\cC_\uu(Q))
  = \max\{\|\ux_j^{(i)}\|,\, |\ux_j^{(i)}\cdot\uu|Q\}
 \quad (1\le j\le n).
\end{equation}
For each
$j\in\{1,\dots,n\}$, we have
\begin{equation}
 \label{inv:prop2:eq2}
 |\ux_j^{(i)}\cdot(\uu-\uu_i)|
  \le \|\ux_j^{(i)}\|\,\|\uu-\uu_i\|
  \le \frac{8e^4A_j^{(i)}}{Q_{i+1}}\,.
\end{equation}
When $j\neq k_i$, the vector $\uu_i$ is orthogonal to
$\ux_j^{(i)}$.  Then, since $Q<Q_{i+1}$, we obtain
\[
 |\ux_j^{(i)}\cdot\uu|Q \le 8e^4A_j^{(i)}
\]
and the inequalities 1) follow from \eqref{inv:prop2:eq1}.

Since $(\ux_1^{(i)},\dots,\ux_n^{(i)})$
is a basis of $\bZ^n$ and since its subsequence obtained
by deleting $\ux_{k_i}^{(i)}$ is almost orthogonal (and
orthogonal to $\uu_i$), Lemmas \ref{dist:lem1} and
\ref{dist:lemH} give
\[
 \frac{\|\ux_{k_i}^{(i)}\|}
          {\|\ux_1^{(i)}\|\cdots\|\ux_n^{(i)}\|}
 \le
 |\ux_{k_i}^{(i)}\cdot\uu_i|
 = \frac{1}{H\big(\big\langle\ux_1^{(i)},\dots,
              \widehat{\ux_{k_i}^{(i)}},\dots,\ux_n^{(i)}
              \big\rangle_\bR\big)}
 \le \frac{e^2\|\ux_{k_i}^{(i)}\|}
          {\|\ux_1^{(i)}\|\cdots\|\ux_n^{(i)}\|},
\]
and therefore
\begin{equation}
 \label{inv:prop2:eq3}
 \frac{A_{k_i}^{(i)}}{2^{n-1}Q_i}
 \le
 |\ux_{k_i}^{(i)}\cdot\uu_i|
 \le
 \frac{e^2A_{k_i}^{(i)}}{Q_i}.
\end{equation}
If $i+1<s$, we also note that
\[
 \frac{Q_{i+1}}{Q_i}
  = \frac{A_{\ell_{i+1}}^{(i+1)}}{A_{k_i}^{(i)}}
  \ge \frac{A_{\ell_{i+1}}^{(i+1)}}{A_{\ell_{i+1}}^{(i)}}
  \ge C
  = 2^{n+3}e^4,
\]
and so \eqref{inv:prop2:eq2} yields
\[
 |\ux_{k_i}^{(i)}\cdot\uu - \ux_{k_i}^{(i)}\cdot\uu_i|
 = |\ux_{k_i}^{(i)}\cdot(\uu-\uu_i)|
 \le \frac{A_{k_i}^{(i)}}{2^nQ_i}\,.
\]
This inequality also holds if $i+1=s$ because in that
case $\uu=\uu_i$.  So, using \eqref{inv:prop2:eq3}, we deduce that
\[
 \frac{A_{k_i}^{(i)}}{2^nQ_i}
 \le
 |\ux_{k_i}^{(i)}\cdot\uu|
 \le
 \frac{8A_{k_i}^{(i)}}{Q_i}.
\]
The estimates 2) then follow from \eqref{inv:prop2:eq1}
with $j=k_i$ because $Q\ge Q_i$.
\end{proof}

\medskip
We conclude this section with the following result which
establishes, in quantitative form, the second assertion of
Theorem \ref{intro:thm} for rigid $n$-systems with sufficiently
large mesh.

\begin{theorem}
\label{inv:thm}
Let $\delta \ge 4+(n+3)\log2$ and let $\uP\colon [q_0,\infty)\to\bR^n$
be a rigid $n$-system with mesh $\delta$.  Then, there exists
a unit vector $\uu\in\bR^n$ such that
\[
 \sup_{q\ge q_0}\|\uP(q)-\uL_\uu(q)\|_\infty \le n\log(8e^4n).
\]
\end{theorem}

\begin{proof}
Consider the sequence of points $(\ua^{(i)})_{0\le i<s}$
and the sequences of integers $(k_i)_{0\le i<s}$ and
$(\ell_i)_{0\le i<s}$ which form the canvas
attached to $\uP$ as in Definitions \ref{intro:def:canvas}
and \ref{intro:def:systems}.
For each integer $i$ with $0\le i<s$, we write
$\ua^{(i)}=(a_1^{(i)},\dots,a_n^{(i)})$ and set
\[
 A_j^{(i)}:=\exp(a_j^{(i)})
 \quad
 (1\le j\le n).
\]
Then, the conditions (C1)--(C3) of Definition \ref{intro:def:canvas}
yield \eqref{inv:H1}--\eqref{inv:H4} and thus Propositions
\ref{inv:prop1} and \ref{inv:prop2} apply.   Consider the bases
$(\ux_1^{(i)},\dots,\ux_n^{(i)})$
of $\bZ^n$ provided by the first proposition for $0\le i<s$
and the unit vector $\uu$ of $\bR^n$ provided by the second one.
We claim that this unit vector has the required property.
To show this, we first note that, in agreement with
Definition \ref{intro:def:systems} and Proposition \ref{inv:prop2},
we have
\[
 q_i=a_1^{(i)}+\cdots+a_n^{(i)}=\log Q_i
 \quad
 (0\le i<s)
\]
and $q_s=\log Q_s=\infty$ if $s\neq\infty$.  Then, we fix
an arbitrary integer $i$ with $0\le i<s$ and a real number
$q\in [q_i,q_{i+1})$.  We also set $Q=e^q$, so that
$Q\in [Q_i,Q_{i+1})$.  Since $(\ux_1^{(i)},\dots,\ux_n^{(i)})$
is a basis of $\bZ^n$, the successive minima of $\cC_\uu(Q)$
are bounded above by a permutation of the numbers
$\lambda(\ux_j^{(i)},\cC_\uu(Q))$ ($1\le j\le n$), and so,
in view of Proposition
\ref{inv:prop2}, they are bounded above by a permutation
of the $n$ numbers $8A_{k_i}^{(i)}Q/Q_i$ and $8e^4A_j^{(i)}$
($1\le j\le n$, $j\neq k_i$).  Taking
logarithms, we deduce that the numbers
$L_{\uu,j}(q):=\log\lambda_j(\cC_\uu(e^q))$ ($1\le j\le n$) are
bounded above by a permutation of $a_{k_i}^{(i)}+q-q_i+c$
and $a_j^{(i)}+c$ ($1\le j\le n$, $j\neq k_i$), where
$c=\log(8e^4)$.  In view of the formula for $\uP(q)$
in Definition \ref{intro:def:systems}, this means that
\[
 L_{\uu,j}(q) \le P_j(q)+c
 \quad
 (1\le j\le n).
\]
where $P_j$ denotes the $j$-th component of $\uP$.
By \eqref{conv:eq:L1+...+Ln} and \eqref{conv:eq:P_1+...+P_n},
we also have
\[
 q-n\log(n) \le \sum_{j=1}^n L_{\uu,j}(q)
 \et
 \sum_{j=1}^n \big(P_j(q)+c\big) = q +nc.
\]
We conclude that $0\le (P_j(q)+c)-L_{\uu,j}(q) \le n(\log(n)+c)$
for each $j=1,\dots,n$ and so
\[
\|\uL_\uu(q)-\uP(q)\|_\infty
  %= \max_{1\le j\le n}|L_{\uu,j}(q)-P_j(q)|
  \le n(\log(n)+c)
  =n\log(8e^4n).
\qedhere
\]
\end{proof}

%%%%%%%%%%%%%%%%%%%%%%%%%%%%%%%%%%%%%%%%%%%%%%%%%%%%%%%%
%
%    Reduced systems
%
%%%%%%%%%%%%%%%%%%%%%%%%%%%%%%%%%%%%%%%%%%%%%%%%%%%%%%%%

\section{Reduced systems}
\label{sec:red}

In this section and the next one, we now turn to the problem
of approximating a given $(n,\gamma)$-system on $[0,\infty)$
by rigid $n$-systems of sufficiently large mesh, for $n\ge 2$.
In the current section, we simply narrow the class of
$(n,\gamma)$-systems that we need to consider.

\begin{definition}
Let $q_0\ge 0$.  An $(n,\gamma)$-reduced system on $[q_0,\infty)$
is an $(n,\gamma)$-system $\uP=(P_1,\dots,P_n)\colon [q_0,\infty) \to \bR^n$
with the property that, for any $j=1,\dots,n-1$, any $a\ge q_0$
and any $b\ge a+n\gamma$ such that $P_1+\cdots+P_j$ is constant
on $[a,b]$, the functions $P_1,\dots,P_j$ are constant on $[a,b-n\gamma]$.
\end{definition}

For example, any $(n,0)$-system is already an $(n,0)$-reduced
system.  The goal of this section is to prove the following result.

\begin{proposition}
\label{red:prop}
Let $\uP\colon[0,\infty)\to\bR^n$ be an $(n,\gamma)$-system
on $[0,\infty)$.  There exists an $(n,2n\gamma)$-reduced system
$\tuP\colon[0,\infty)\to\bR^n$ such that
$\|\uP-\tuP\|_\infty\le n\gamma$.
\end{proposition}

The proof goes through several steps, based on
the following observation whose proof is left
to the reader.

\begin{lemma}
\label{red:lem1}
Let $a,b\in\bR$ with $a<b$, and let $M\colon[a,b]\to \bR$
be a continuous piecewise linear function with slopes $0$
and $1$.  Then the map $\tM\colon[a,b]\to\bR$ given by
\[
 \tM(q)=\max\{M(a),M(b)+q-b\}
 \quad
 \text{for any $q\in[a,b]$}
\]
is continuous and piecewise linear with slopes
$0$ and $1$.  Moreover, it satisfies
\begin{equation}
\label{red:lem:eq}
 M(a)=\tM(a)\le \tM(q)\le M(q)\le M(b)=\tM(b)
\end{equation}
for each $q\in[a,b]$.
\end{lemma}

% \begin{proof}
% The first assertion is clear.
% Since $M$ has slopes $0$ and $1$, the map sending $q\in[a,b]$ to
% $M(q)-q+b$ has slopes $0$ and $-1$.  So, it is monotone decreasing
% on $[a,b]$ while $M$ is monotone increasing on the same
% interval.  For $q\in[a,b]$, this yields the inequalities
% $M(b)\le M(q)-q+b$ and $M(a)\le M(q)$, thus $\tM(q)\le M(q)$.
% Clearly, we also have $M(q)\le M(b)=\tM(b)$ and
% $M(a)\le \tM(q)$.  Upon choosing $q=a$, we deduce that
% $M(a)=\tM(a)$.
% \end{proof}

Write $\uP=(P_1,\dots,P_n)$ and set $M_j=P_1+\cdots+P_j$
for each $j=0,\dots,n$.  We define recursively a sequence of
continuous piecewise linear functions $\tM_n,\dots,\tM_1$ from
$[0,\infty)$ to $\bR$ with slopes $0$ and $1$ in the following way.
We first set $\tM_n=M_n$ and $\tM_{n-1}=M_{n-1}$.  Then, assuming that
$\tM_j$ has been constructed for some index $j$ with $2\le j\le n-1$,
we form the set $\cE_j$ of all maximal sub-intervals
of $[0,\infty)$ with non-empty interior on which $\tM_j$ is constant.
The end-points of these intervals are the points of $(0,\infty)$ where
$\tM_j$ is not differentiable, and possibly the point $0$.  So, they
form a discrete subset of $[0,\infty)$.  We define $\tM_{j-1}\colon
[0,\infty)\to\bR$ by
\[
 \tM_{j-1}(q)=
 \begin{cases}
  \max\{M_{j-1}(a),M_{j-1}(b)+q-b\}
    &\text{if $q\in[a,b]\in\cE_j$,}\\
  M_{j-1}(a)
    &\text{if $q\in[a,\infty)\in\cE_j$,}\\
  M_{j-1}(q)
    &\text{if $q\notin I$ for any $I\in\cE_j$,}
 \end{cases}
\]
Lemma \ref{red:lem1} applied to the restriction
of $M_{j-1}$ to any bounded interval $[a,b]\in\cE_j$ shows
that $\tM_{j-1}$ and $M_{j-1}$ agree at the end-points of
such an interval.  Moreover $\tM_{j-1}$ and $M_{j-1}$
agree at the point $a$ if $[a,\infty)\in\cE_j$.
Therefore, $\tM_{j-1}$ is continuous.  Clearly it is
piecewise linear with slopes $0$ and $1$.  Once
$\tM_n,\dots,\tM_1$ have been constructed, we set
$\tM_0=0$.

\begin{lemma}
\label{red:lem2}
Let $j\in\{1,\dots,n-1\}$.
\begin{itemize}
\item[(i)]
  For any $q\ge 0$, we have $\tM_j(q)\le M_j(q)\le \tM_j(q)+(n-j-1)\gamma$.
\item[(ii)] Suppose $j\ge 2$.  Then, for any compact
  sub-interval $[c,d]$ of $[0,\infty)$ on which $\tM_j$ is
  constant, we have $0\le M_{j-1}(d)-\tM_{j-1}(c)\le (n-j)\gamma$.
\end{itemize}
\end{lemma}

\begin{proof}
We proceed by descending induction on $j$.  For $j=n-1$,
the assertion (i) is clear because $\tM_{n-1}=M_{n-1}$.
If $n=2$, there is nothing more to prove.  Suppose that
$n\ge 3$ and that (i) holds for some integer $j \in
\{2,\dots,n-1\}$.  Let $[c,d]$ be a compact subinterval of
$[0,\infty)$ on which $\tM_j$ is constant.  If $[c,d]$ is not
contained in any interval of $\cE_j$, we must have $c=d$ and
$\tM_{j-1}(c)=M_{j-1}(c)$, so (ii) holds.  Otherwise
$[c,d]$ is contained in an interval $I$ of $\cE_j$ of the form
$[a,b]$ or $[a,\infty)$.  In the first case, Lemma
\ref{red:lem1} shows that $\tM_{j-1}(a)=M_{j-1}(a)$ and
$\tM_{j-1}(c)\le M_{j-1}(c)$.  This is also true in the
second case.  Then,
combining this with the fact that $M_{j-1}$ and $\tM_{j-1}$
are monotone increasing, we deduce that
\[
 M_{j-1}(d)-\tM_{j-1}(c)\ge M_{j-1}(c)-\tM_{j-1}(c) \ge 0
\]
and
\[
\begin{aligned}
 M_{j-1}(d)-\tM_{j-1}(c)
   &\le M_{j-1}(d)-\tM_{j-1}(a)\\
   &= M_{j-1}(d)-M_{j-1}(a)\\
   &= (M_j(d)-M_j(a))+(P_j(a)-P_j(d)).
\end{aligned}
\]
Since $\uP$ is an $(n,\gamma)$-system and since $a\le d$ we also have
$P_j(a)-P_j(d)\le \gamma$.  Moreover, the induction hypothesis gives
$M_j(d)\le \tM_j(d)+(n-j-1)\gamma$.  As $\tM_j$ is constant on
$I \supseteq [a,d]$, we further have $\tM_j(d) = \tM_j(a) = M_j(a)$.
This means that $M_j(d)-M_j(a)\le (n-j-1)\gamma$.  So, we conclude that
\[
 0\le M_{j-1}(d)-\tM_{j-1}(c)\le (n-j)\gamma
\]
and therefore (ii) holds for that value of $j$.  Applying this result with
$c=d=q$ for an arbitrary $q\ge 0$, we obtain $\tM_{j-1}(q)\le M_{j-1}(q)\le
\tM_{j-1}(q)+(n-j)\gamma$.  Thus (i) holds with $j$ replaced by $j-1$.
This completes the induction step and proves the lemma.
\end{proof}

\begin{lemma}
\label{red:lem3}
Let $j\in\{1,\dots,n-1\}$. Suppose that $\tM_j$ is constant
on some compact sub-interval $[c,d]$ of $[0,\infty)$ with
$d\ge c+n\gamma$.  Then $\tM_{j-1}$ is constant on $[c,d-n\gamma]$.
\end{lemma}

\begin{proof}
If $j=1$, this is true because $\tM_0=0$.  Suppose that $j\ge 2$.
Then, $[c,d]$ is contained in some interval of $\cE_j$.  If this
interval is of the form $[a,\infty)$, then $\tM_{j-1}$ is
constant on $[a,\infty)$ and so it is constant on $[c,d]$.
Otherwise this interval is of the form $[a,b]$ and, on this
interval, $\tM_{j-1}$ has slope $0$ and then slope $1$.
If $\tM_{j-1}$ is not constant on $[c,d-n\gamma]$, this means
that it has slope $1$ on $[d-n\gamma,b]$ and so
we obtain
\[
 n\gamma = \tM_{j-1}(d)-\tM_{j-1}(d-n\gamma)
         \le \tM_{j-1}(d)-\tM_{j-1}(c).
\]
This is impossible because Lemma \ref{red:lem2} gives
\[
 \tM_{j-1}(d)-\tM_{j-1}(c)
  \le M_{j-1}(d)-\tM_{j-1}(c)
  \le (n-j)\gamma.
\qedhere
\]
\end{proof}

For each $j=1,\dots,n$, we define $\tP_j=\tM_j-\tM_{j-1}$ so that
\[
 \tM_j=\tP_1+\cdots+\tP_j \quad \text{for $j=1,\dots,n$.}
\]
We claim that the resulting map $\tuP=(\tP_1,\dots,\tP_n)
\colon [0,\infty)\to\bR^n$ satisfies all the conditions of
Proposition \ref{red:prop}.  It clearly satisfies the conditions
(S3) and (S4) of Definition \ref{Conv:def:SS}.  The next two
lemmas show that it satisfies all the other requirements.

\begin{lemma}
\label{red:lem4}
We have
\begin{itemize}
 \item[(i)] $|P_j(q)-\tP_j(q)|\le (n-1)\gamma$
  \quad $(1\le j\le n,\ 0\le q)$,
 \smallskip
 \item[(ii)] $-n\gamma \le \tP_j(q)\le \tP_{j+1}(q)+2n\gamma$
  \quad $(1\le j< n,\ 0\le q)$,
 \smallskip
 \item[(iii)] $\tP_j(q_1)\le \tP_j(q_2)+2n\gamma$
  \quad $(1\le j\le n,\ 0\le q_1\le q_2)$.
\end{itemize}
Moreover, if, for some $j\in\{1,\dots,n-1\}$, the function
$\tM_j$ is constant on some compact subinterval
$[c,d]$ of $[0,\infty)$ with $d\ge c+n^2\gamma$, then
$\tP_1,\dots,\tP_j$ are constant on $[c,d-n^2\gamma]$.
\end{lemma}

\begin{proof}
Let $q\ge 0$.  Lemma \ref{red:lem2} (i) shows that
$0\le M_j(q)-\tM_j(q)\le (n-1)\gamma$ for $j=1,\dots,n-1$.
This is also true for $j=0$ and $j=n$ because
in those cases we have $M_j=\tM_j$.  Then Part (i)
follows from the equalities
\[
 P_j(q)-\tP_j(q) = (M_j(q)-\tM_j(q))-(M_{j-1}(q)-\tM_{j-1}(q))
 \quad
 (1\le j \le n).
\]
Let $j\in\{1,\dots,n\}$.  If $j< n$, the inequality
$-\gamma\le P_j(q)\le P_{j+1}(q)+\gamma$ combined with the estimates
of Part (i) yields
\begin{align*}
 \tP_j(q)
   &\ge P_j(q)-(n-1)\gamma
   \ge -n\gamma,\\
 \tP_j(q)
   &\le P_j(q)+(n-1)\gamma
   \le P_{j+1}(q)+n\gamma
   \le \tP_{j+1}(q)+(2n-1)\gamma.
\end{align*}
Similarly, if $0\le q_1\le q_2$, the inequality
$P_j(q_1)\le P_j(q_2)+\gamma$ yields
\[
 \tP_j(q_1)
   \le P_j(q_1)+(n-1)\gamma
   \le P_j(q_2)+n\gamma
   \le \tP_j(q_2)+(2n-1)\gamma.
\]
This proves (ii) and (iii).  Finally, if $\tM_j$ is constant
on $[c,d]$ for some $j\in\{1,\dots,n-1\}$ and some compact
subinterval $[c,d]$ of $[0,\infty)$ of length at least
$n^2\gamma$, then Lemma \ref{red:lem3} shows that
$\tM_j,\dots,\tM_1$ are constant on $[c,d-nj\gamma]$
and so $\tP_1,\dots,\tP_j$ are also constant on the latter
interval.
\end{proof}

\begin{lemma}
\label{red:lem5}
Let $j\in\{2,\dots,n\}$.  Suppose that $\tM_{j-1}$ changes slope
from $1$ to $0$ at a point $q>0$. Then we have
$\tP_j(q) \le \tP_{j-1}(q)+2n\gamma$.
\end{lemma}

\begin{proof}
Suppose first that $\tM_{j-1}$ coincides with $M_{j-1}$ on some
open neighborhood of $q$.  Then, as $M_{j-1}$ changes slope
from $1$ to $0$ at $q$, we have $P_j(q) \le P_{j-1}(q)+\gamma$
and the conclusion follows from Lemma \ref{red:lem4} (i).

Suppose now that such a neighborhood does not exist.
Then $q$ belongs to an interval $I$ of $\cE_j$ and, since
$\tM_{j-1}$ changes slope at most once on $I$,
going from slope $0$ to slope $1$, the point $q$ lies on the
boundary of $I$.  We claim that there is an interior point
$p$ of $I$ at which $M_{j-1}$ changes slope from $1$ to $0$.
If $I=[a,b]$ is a compact interval, this is clear because
otherwise $M_{j-1}$ would agree with $\tM_{j-1}$
on the whole interval $[a,b]$ and therefore would agree on an
open interval containing $[a,b]$. If instead
$I=[a,\infty)$ is unbounded, then $q=a$ and, since $M_{j-1}$
is constant on $[a,\infty)$, the function $M_{j-1}$, by differing
locally from $\tM_{j-1}$, should have slope $1$ in a right
neighborhood $[a,a+\epsilon)$
of $a$, for some $\epsilon>0$.  However, $M_{j-1}$ cannot
have constant slope $1$ on $[a,\infty)$ because $\tM_{j-1}$ is
constant on $[a,\infty)$ and, by Lemma \ref{red:lem2} (i),
the difference $M_{j-1}-\tM_{j-1}$ is bounded. Thus
the claim also holds in that case as well.
By definition of an $(n,\gamma)$-system, we have
\[
 P_j(p) \le P_{j-1}(p)+\gamma.
\]
Put $c=\min\{p,q\}$ and $d=\max\{p,q\}$.  Since $\tM_j$ is
constant on $I \supseteq [c,d]$, we find, with the help of
Lemma \ref{red:lem2},
\[
 \begin{aligned}
 \tP_j(q)-P_j(p)
   &= (\tM_j(q)-M_j(p))+(M_{j-1}(p)-\tM_{j-1}(q))\\
   &\le (\tM_j(p)-M_j(p))+(M_{j-1}(d)-\tM_{j-1}(c))
   \le (n-j)\gamma.
 \end{aligned}
\]
Since $p\le d$, we also have $P_{j-1}(p) \le
P_{j-1}(d)+\gamma$, and thus
\[
 \begin{aligned}
  P_{j-1}(p)-\tP_{j-1}(q)
   &\le P_{j-1}(d)-\tP_{j-1}(q) +\gamma\\
   &= (M_{j-1}(d)-\tM_{j-1}(q))+(\tM_{j-2}(q)-M_{j-2}(d))+\gamma\\
   &\le (M_{j-1}(d)-\tM_{j-1}(c))+(\tM_{j-2}(d)-M_{j-2}(d))+\gamma\\
   &\le (n-j+1)\gamma.
 \end{aligned}
\]
Combining the three displayed estimates, we get
$\tP_j(q)-\tP_{j-1}(q) \le 2(n-j+1)\gamma \le 2n\gamma$.
\end{proof}

%%%%%%%%%%%%%%%%%%%%%%%%%%%%%%%%%%%%%%%%%%%%%%%%%%%%%%%
%
%   Approximation
%
%%%%%%%%%%%%%%%%%%%%%%%%%%%%%%%%%%%%%%%%%%%%%%%%%%%%%%%

\section{Approximation by rigid $n$-systems}
\label{sec:app}

We shall now prove the following result.

\begin{proposition}
 \label{app:prop}
Let $\gamma,\delta\in\bR$ with $0\le \gamma< \delta/(2n^2)$
and let $\uP\colon[0,\infty)\to \bR^n$ be an $(n,\gamma)$-reduced
system.  Put $q_0=n(n+1)\delta/2$.  Then there exists a
rigid $n$-system $\uR\colon[q_0,\infty)\to\bR^n$ of mesh $\delta$
such that $\|\uP(q)-\uR(q)\|_\infty\le 3n^2\delta$ for each $q\ge q_0$.
\end{proposition}

We first note that a change of variables reduces the proof to
the case where $\delta=1$.  Indeed, suppose for the moment that
the proposition holds in that case and let us use the more suggestive
terminology \emph{rigid integral $n$-system} to denote a rigid
$n$-system of mesh $1$.  Under the hypotheses of the proposition,
we form the map $\tuP\colon[0,\infty)\to\bR^n$
given by $\tuP(q)=\delta^{-1}\uP(q\delta)$ for each $q\ge 0$.  Then
$\tuP$ is an $(n,\gamma/\delta)$-reduced system on $[0,\infty)$.
As $\gamma/\delta<1/(2n^2)$, there exists a rigid integral
$n$-system $\tuR\colon[q_0,\infty)\to\bR^n$ with $q_0=n(n+1)/2$ such
that $\|\tuP(q)-\tuR(q)\|_\infty \le 3n^2$ for each $q\ge q_0$. The map
$\uR\colon[q_0\delta,\infty)\to \bR^n$ given by
$\uR(q)=\delta\tuR(q/\delta)$ for each $q\ge q_0\delta$ is then a rigid
$n$-system of mesh $\delta$ which satisfies $\|\uP(q)-\uR(q)\|_\infty
\le 3n^2\delta$ for each $q\ge q_0\delta$ as requested.

\smallskip
So, from now on, we assume that $\delta=1$.  We fix a real number
$\gamma$ with $0\le \gamma < 1/(2n^2)$ and an $(n,\gamma)$-reduced
system $\uP=(P_1,\dots,P_n)$ on $[0,\infty)$. For each $j\in\{1,\dots,n\}$,
we define a new function $\Pbar_j\colon[0,\infty)\to\bR$ by putting
\[
 \Pbar_j(q) = \gamma+\sup\{P_j(t)\,;\,0\le t\le q\} \quad (q\ge 0).
\]
Then $\Pbar_1,\dots,\Pbar_n$ are continuous
piecewise linear functions with slopes $0$ and $1$.
The conditions (S1) and (S2) of Definition \ref{Conv:def:SS}
respectively imply that, for each $q\ge 0$, they satisfy
\begin{align}
 \label{app:eq:PPbarjj}
 0&\le \Pbar_j(q)\le \Pbar_{j+1}(q) +\gamma
 \qquad(1\le j< n),\\
 \label{app:eq:PPbarj}
 P_j(q)+\gamma&\le \Pbar_j(q)\le P_j(q) +2\gamma
 \ \qquad(1\le j\le n).
\end{align}
Moreover, we note that $\Pbar_1=\gamma+P_1$ and $\Pbar_n=\gamma+P_n$ since
$P_1$ and $P_n$ do not take negative slope.

We also define recursively a sequence of functions $E_1,\dots,E_n$
from $[0,\infty)$ to $\bN^*$ by putting, for each $q\ge 0$,
\[
 \begin{aligned}
 E_1(q)&= \lfloor \Pbar_1(q) \rfloor + 1,\\
 E_j(q)&= \max\big\{E_{j-1}(q) + 1,\,
         \lfloor \Pbar_j(q)-2(j-1)\gamma \rfloor + 1\big\}
   && \text{for $j=2,\dots,n$,}
 \end{aligned}
\]
where $\lfloor x\rfloor$ stands for the integral part of a real number $x$.
Then each $E_j$ is monotone increasing and right continuous.
Moreover, the set of points of discontinuity of $E_j$ is discrete
and, at such a point $q$, we have $E_j(q)=E_j(q^-)+1$, where $E_j(q^-)$
is a shorthand for $\lim_{t\to q^-}E_j(t)$.  These functions make up a map
\[
 \uE=(E_1,\dots,E_n)\colon[0,\infty)\to(\bN^*)^n
\]
whose values form strictly increasing sequences of positive
integers $0<E_1(q)<\cdots<E_n(q)$.  Condition (S4) of Definition
\ref{Conv:def:SS} together with \eqref{app:eq:PPbarj} and the fact
that each $\Pbar_j$ has slopes $0$ and $1$ implies that
\[
 \Pbar_j(q)
  \le q+\Pbar_j(0)
  \le q+\sum_{k=1}^n(P_k(0)+2\gamma)
  = q+2n\gamma<1
 \quad (1\le j\le n,\, 0\le q \le 1/2),
\]
and so we obtain
\begin{equation}
 \label{app:eq:E0}
 \uE(q)=(1,2,\dots,n)
 \quad (0\le q \le 1/2).
\end{equation}
The next lemma compares the functions $E_j$ and $\Pbar_j$.

\begin{lemma}
\label{app:lem1}
For each $q\ge 0$, we have
\begin{itemize}
 \item[(i)] $\Pbar_j(q)-2(j-1)\gamma < E_j(q) \le \Pbar_j(q)+j(1+\gamma)$
   \ for $j=1,\dots,n$,
 \smallskip
 \item[(ii)] $\|\uE(q)-\uP(q)\|_\infty \le n+1$,
 \smallskip
 \item[(iii)] $|(E_1(q)+\cdots+E_n(q))-q| \le n(n+1)$.
\end{itemize}
\end{lemma}

\begin{proof}
Fix $q\ge 0$.  For $j=1$, the inequality (i) is clear
because $E_1(q)$ is the smallest integer which is greater
than $\Pbar_1(q)$. Suppose that this inequality holds
for some integer $j$ with $1\le j <n$.  We find
\[
 \Pbar_{j+1}(q)-2j\gamma
   < \lfloor \Pbar_{j+1}(q)-2j\gamma \rfloor + 1
   \le E_{j+1}(q)
\]
and, using the induction hypothesis together with
\eqref{app:eq:PPbarjj}, we obtain
\begin{align*}
 E_{j+1}(q)
    &\le \max\{E_j(q)+1,\, \Pbar_{j+1}(q)-2j\gamma+1\} \\
    &\le \max\{\Pbar_j(q)+j(1+\gamma)+1,\, \Pbar_{j+1}(q)+1\}\\
    &\le \Pbar_{j+1}(q)+(j+1)(1+\gamma).
\end{align*}
So, the inequality (i) is also satisfied with $j$ replaced
by $j+1$. Consequently, it holds for $j=1,\dots,n$.  Using
\eqref{app:eq:PPbarj} and the fact that $\gamma\le 1/(2n^2)$,
this implies that
\[
 |E_j(q)-P_j(q)|
  \le \max\{2(j-1)\gamma,\,j(1+\gamma)+2\gamma\}
  \le n+1
 \quad (1\le j\le n),
\]
which proves (ii).
Then (iii) follows because $P_1(q)+\cdots+P_n(q)=q$.
\end{proof}

\begin{lemma}
\label{app:lem2}
Let $(a,b)$ be an open sub-interval of\/ $[0,\infty)$ of
length $b-a\le 1$.  Each of the functions $E_1,\dots,E_n$
admits at most one point of discontinuity on $(a,b)$.
\end{lemma}

\begin{proof}
Suppose on the contrary that, for some $j\in
\{1,\dots,n\}$, the function $E_j$ admits at least two points
of discontinuity $q_1<q_2$ on $(a,b)$.  Let $j$ be the
minimal index with this property and choose $q_1,q_2$ so
that $E_j$ is constant on $(q_1,q_2)$.  Then, we have
\[
 E_j(q_1^-)=H-1,
 \quad
 E_j(q_1)=E_j(q_2^-)=H
 \et
 E_j(q_2)=H+1
\]
for some $H\in\bZ$.  Since $\Pbar_j$ is continuous,
we deduce from Lemma \ref{app:lem1} that
\[
 \Pbar_j(q_1)
  \le E_j(q_1^-)+2(j-1)\gamma = H-1+2(j-1)\gamma.
\]
As $\Pbar_j$ has slope at most $1$ and as $q_2-q_1<b-a\le 1$,
this gives
\[
 \Pbar_j(q_2) < \Pbar_j(q_1)+1 \le H+2(j-1)\gamma,
\]
thus $\lfloor \Pbar_j(q_2)-2(j-1)\gamma \rfloor + 1 \le H$.
Since $E_j(q_2)=H+1$, this implies that $j\ge 2$ and that
$E_{j-1}$ is discontinuous at the point $q_2$, with
$E_{j-1}(q_2)=H$.  By virtue of the choice of $j$, this
point $q_2$ must be the only point of discontinuity of $E_{j-1}$
on $(a,b)$.  So, $E_{j-1}$ is constant on $(a,q_2)$, equal
to $E_{j-1}(q_2^-)=H-1$.  This means in particular that
$E_{j-1}(q_1^-)=H-1=E_j(q_1^-)$, a contradiction.
\end{proof}

Let $\Sigma$ denote the set of all points of discontinuity
of $\uE$ in $[0,\infty)$.  Since, by \eqref{app:eq:E0},
$\uE$ is constant on $[0,1/2]$, the above lemma
shows that $\Sigma$ is a discrete subset of $[1/2,\infty)$.
It is also infinite because, by Lemma \ref{app:lem1} (iii),
the sum $E_1+\cdots+E_n$ is unbounded while it increases
by at most $n$ at each point of discontinuity of $\uE$.
We now study
\[
 \Sigmabar
   := \Sigma+[-n\gamma,n\gamma]
   =\{ t\in\bR\,;\,
      \text{$|t-q|\le n\gamma$ for some $q\in\Sigma$}\}
   \subset [0,\infty).
\]

\begin{lemma}
\label{app:lem3}
The set $\Sigmabar$ has infinitely many connected components.
Any such component $I$ is a closed interval of length
at least $2n\gamma$ and at most $2n^2\gamma<1$, on which
each of the functions $E_1,\dots,E_n$
admits at most one point of discontinuity.
\end{lemma}

\begin{proof}
Since $\Sigma$ is a discrete subset of $\bR$,
the set $\Sigmabar$ is a union of disjoint closed
intervals of length at least $2n\gamma$.
Let $I$ be one of these, and let $[a,b]$ be a compact
subinterval of $I$.  The intersection of $\Sigma$ with
$[a-n\gamma,b+n\gamma]$ consists of finitely many points
$q_1<\cdots<q_s$.  For these points, we have
\begin{equation}
\label{app:lem3:eq}
 [a,b] \subseteq
   [q_1-n\gamma,q_1+n\gamma] \cup\cdots\cup [q_s-n\gamma,q_s+n\gamma]
\end{equation}
and $q_{i+1}\le q_i+2n\gamma$ for $i=1,\dots,s-1$.  If
$s\ge n+1$, then one of the functions $E_1,\dots,E_n$ admits
at least two points of discontinuity among $q_1,\dots,q_{n+1}$.
By Lemma \ref{app:lem2}, this is impossible because
$[q_1,q_{n+1}]$ has length $q_{n+1}-q_1\le 2n^2\gamma<1$.
Thus, we must have $s\le n$, and the inclusion \eqref{app:lem3:eq}
yields $b-a\le 2ns\gamma\le 2n^2\gamma$.  This shows that $I$ is
bounded of length at most $2n^2\gamma<1$ and so, again by Lemma
\ref{app:lem2}, each of the functions $E_1,\dots,E_n$ has
at most one point of discontinuity in $I$.  Finally, $\Sigmabar$
consists of infinitely many disjoint such intervals because
it contains $\Sigma$ which is infinite and discrete.
\end{proof}

In view of the above lemma, we can write
\[
 \Sigmabar=[c_0,d_0]\cup[c_1,d_1]\cup\cdots
\]
for an infinite sequence of real numbers
$0\le c_0<d_0<c_1<d_1<\cdots$ with
\[
 2n\gamma\le d_i-c_i < 1
 \quad (i\ge 0).
\]
For each $i\ge 0$, we denote by $\Omega^{(i)}$ the
non-empty set of indices $j\in\{1,\dots,n\}$ for
which $E_j$ is not constant on $[c_i,d_i]$, and we define
\begin{equation}
 \label{app:eq:kl}
 \begin{aligned}
 k^{(i)}&:=\max\{j\in\Omega^{(i)}\,;\,
    \text{$j=1$ or $E_{j-1}(d_i) < E_j(d_i)-1$}\,\},\\
 \quad
 \ell^{(i)}&:=\,\min\{j\in\Omega^{(i)}\,;\,
    \text{$j=n$ or $E_{j+1}(c_i) > E_j(c_i)+1$}\,\}.
 \end{aligned}
\end{equation}
Our next goal is to show that $\ell^{(i)}\ge k^{(i+1)}$
for each $i\ge 0$.  This is the most delicate point.
Once it is settled, the construction of the requested rigid
integral $n$-system goes quickly as the
reader could see by going directly to Lemma \ref{app:lem8}.

To prove the above inequality, we use freely the estimates
\eqref{app:eq:PPbarjj} and \eqref{app:eq:PPbarj}.  For each
$j\in\{1,\dots,n\}$, we set, as usual, $M_j=P_1+\cdots+P_j$. We also
use the fact that, if $j<n$ and if $P_j+\gamma<P_{j+1}$ on a
subinterval $[a,b]$ of $[0,\infty)$, then, by Condition (S5)
of Definition \ref{Conv:def:SS}, the function $M_j$ is concave up
on $[a,b]$: either its slope is constant on $[a,b]$, or it is $0$
on $[a,c]$ and $1$ on $[c,b]$ for some $c\in(a,b)$.
We write $f'(q^-)$ to denote the left derivative
of a function $f$ at a point $q$ and, for each $q\in\Sigma$, we denote
by $\Omega(q)$ the non-empty set of indices $j\in \{1,\dots,n\}$ such that
$E_j$ is discontinuous at $q$.  We first establish two lemmas whose proofs
are illustrated on Figure 3.

\begin{lemma}
\label{app:lem5}
Let $q\in\Sigma$ and $k\in\Omega(q)$.  Suppose that $k>1$
and that $E_{k-1}(q)< E_k(q)-1$.  Then, we have $P_k'(q^-)=1$ and
$P_k(q)=E_k(q^-)+(2k-3)\gamma$.  Moreover, the function $M_{k-1}$
is constant on $[q-1,q]$.
\end{lemma}

\begin{proof}
Set $H=E_k(q^-)$ so that $E_k(q)=H+1$ and $E_{k-1}(q)\le H-1$.
Since $E_{k-1}(q)+1< E_k(q)$, we must have
\[
 \lfloor \Pbar_k(q)-2(k-1)\gamma\rfloor + 1
   = E_k(q) = H+1,
\]
and so $\Pbar_k(q) \ge H+2(k-1)\gamma$.  Moreover, for each
$t\in[0,q)$, Lemma \ref{app:lem1} gives
\[
 \Pbar_k(t)
   < E_k(t)+2(k-1)\gamma
   \le H+2(k-1)\gamma
   \le \Pbar_k(q).
\]
This is incompatible with having either $P_k'(q^-)=0$ or
$P_k(q)+\gamma<\Pbar_k(q)$.  So, we must have $P_k'(q^-)=1$ and,
by letting $t$ tend to $q$ with $t<q$, we conclude that
$P_k(q) = \Pbar_k(q)-\gamma = H+(2k-3)\gamma$.  This proves the first
assertion of the lemma.

Since $E_{k-1}(q)\le H-1$, Lemma \ref{app:lem1} shows that,
for any $t\in[0,q]$, we have
\[
 P_{k-1}(t)
   \le \Pbar_{k-1}(q)
   < E_{k-1}(q)+2(k-2)\gamma
   \le H-1+2(k-2)\gamma.
\]
As $P_k$ is continuous and piecewise linear with slope
at most $1$, we conclude that, for any $t\in[q-1,q]$,
we have
\[
 P_k(t)
   \ge P_k(q)-1
   = H-1+(2k-3)\gamma
   > P_{k-1}(t)+\gamma.
\]
Thus $M_{k-1}$ is concave up on $[q-1,q]$.  However, the inequalities
\[
 0\le M_{k-1}'(q^-) = M_k'(q^-)-P_k'(q^-) \le 1 -P_k'(q^-)=0
\]
show that $M_{k-1}'(q^-)=0$ and therefore $M_{k-1}$ must be constant on
$[q-1,q]$.
\end{proof}

\begin{figure}[h]
\label{fig3}
 \begin{tikzpicture}[xscale=0.55,yscale=0.5]
       \path [thick,fill=gray!30,opacity=0.8] (0,0.2) -- (4,0.2) -- (4,2) -- (0,2) -- (0,0.2);
       \path [thick,fill=gray!30,opacity=0.8] (0,3) -- (4,7) -- (0,7) -- (0,3);
       \draw[dashed] (0, 9) -- (0, 0) node[below]{$q-1$};
       \draw[dashed] (4, 9) -- (4, 0) node[below]{$q$};
       \draw[semithick,dashed,font=\footnotesize] (0,2) -- (4,2) node[right]{$H-1+(2k-4)\gamma$};
       \draw (2,1) node[]{$P_{k-1}$};
       \draw[thin,dotted,font=\footnotesize] (0,3) -- (4,3) node[right]{$H-1+(2k-3)\gamma$};
       \draw[thick,font=\footnotesize] (0,5) -- (4,5) node[right]{$H$};
       \draw (3.3,5) node[below]{$E_k$};
       \node[draw,circle,inner sep=1pt,fill=white] at (4,5) {};
       \draw[semithick,font=\footnotesize] (0,3) -- (4,7) node[right]{$H+(2k-3)\gamma$};
       \draw[semithick,dashed] (0,7) -- (4,7) {};
       \draw (2,6.2) node[]{$P_{k}$};
       \draw[thick,font=\footnotesize] (8,9) -- (4,9) node[left]{$H+1$};
       \draw (6,9) node[below]{$E_k$};
       \node[draw,circle,inner sep=1pt,fill] at (4,9) {};
 \end{tikzpicture}
 \begin{tikzpicture}[xscale=0.55,yscale=0.5]
       \path [thick,fill=gray!30,opacity=0.8] (0,0.2) -- (10,0.2) -- (10,6) -- (4,6) -- (4,2) -- (0,2) --(0,0.2);
       \path [thick,fill=gray!30,opacity=0.8] (0,3) -- (4,7) -- (10,7) -- (10,9) -- (0,9) -- (0,3);
       \draw[dashed] (0, 9) -- (0, 0) node[below]{$q-1$};
       \draw[dashed] (4, 9) -- (4, 0) node[below]{$q$};
       \draw[dashed] (10, 9) -- (10, 0) node[below]{$r$};
       \draw[thick] (0,1.2) -- (4,1.2) {};
       \draw[dotted,thin,font=\footnotesize] (4,1.2) -- (10,1.2) node[right]{$H-1$};
       \node[draw,circle,inner sep=1pt,fill=white] at (4,1.2) {};
       \draw (2,1.2) node[below]{$E_\ell$};
       \draw[semithick,dashed] (0,2) -- (4,2) {};
       \draw[thin,dotted,font=\footnotesize] (4,2) -- (10,2)node[right]{$H-1+(2\ell-3)\gamma$};
       \draw[thick,font=\footnotesize] (4,5.2) -- (10,5.2) node[right]{$H$};
       \node[draw,circle,inner sep=1pt,fill] at (4,5.2) {};
       \draw (7,5.2) node[below]{$E_\ell$};
       \draw[semithick,dashed,font=\footnotesize] (4,6) -- (10,6) node[right]{$H+(2\ell-3)\gamma$};
       \draw[semithick,font=\footnotesize] (0,3) -- (4,7) -- (10,7) node[right]{$H+(2\ell-2)\gamma$};
       \draw[thin,dotted,font=\footnotesize] (0,3) -- (10,3) node[right]{$H-1+(2\ell-2)\gamma$};
       \draw (4.8,8) node[right]{$P_{\ell+1}$};
       \draw (4.8,2.5) node[right]{$P_{\ell}$};
       \end{tikzpicture}
\caption{Illustrations for the proofs of Lemma \ref{app:lem5} on the left, and
of Lemma \ref{app:lem6} on the right.  The shaded regions contain
the graphs of $P_{k-1}$ and $P_k$ on the left, and those of $P_\ell$ and
$P_{\ell+1}$ on the right, showing the gap between them.}
\end{figure}
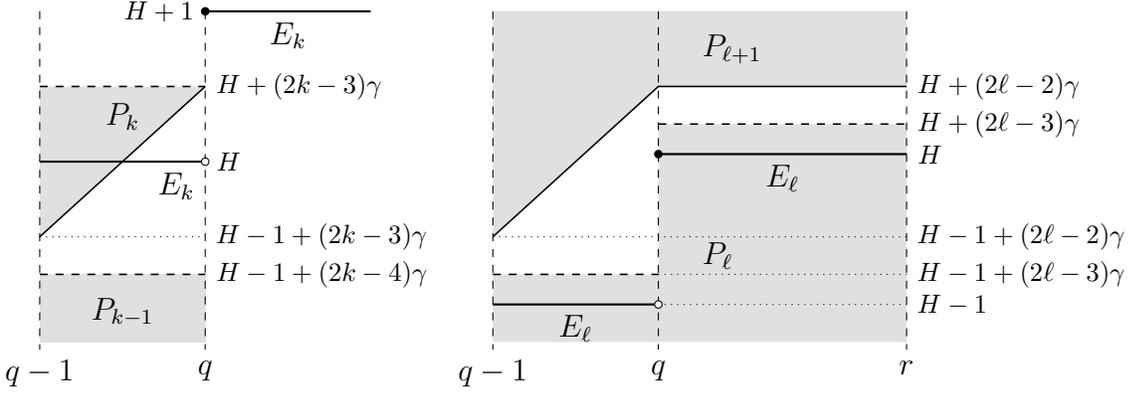

\begin{lemma}
\label{app:lem6}
Let $q\in\Sigma$ and $\ell\in\Omega(q)$.  Suppose that $\ell<n$
and that $E_{\ell+1}(q^-)> E_{\ell}(q^-)+1$.  Choose
$r > q$ such that $E_\ell$ is constant on $[q,r)$.
Then, $M_\ell$ is concave up on the interval $[q-1,r]$.
\end{lemma}

\begin{proof}
Set $H=E_\ell(q)$ so that $E_\ell(q^-)=H-1$ and
$E_{\ell+1}(q^-)\ge H+1$.  Choose
$\epsilon>0$ such that $E_\ell$ and $E_{\ell+1}$ are constant on
$[q-\epsilon,q)$.  For each $t\in[q-\epsilon,q)$, we have
$E_\ell(t)+1 = H < E_{\ell+1}(t)$ and therefore
\[
 H+1
   \le E_{\ell+1}(t)
   = \lfloor \Pbar_{\ell+1}(t)-2\ell\gamma\rfloor + 1
   \le \Pbar_{\ell+1}(t)-2\ell\gamma + 1.
\]
Letting $t$ tend to $q$ and using the fact that $\Pbar_{\ell+1}$ is
continuous, we deduce that
\begin{equation}
 \label{app:lem6:eq}
 \Pbar_{\ell+1}(q)\ge H+2\ell\gamma.
\end{equation}
By Lemma \ref{app:lem2}, the function $E_\ell$ has no point of
discontinuity on $(q-1,q)$.  Since it is right continuous,
it is therefore constant equal to $H-1$ on $[q-1,q)$.  Then,
for any $t\in[q-1,q)$, Lemma \ref{app:lem1} yields
\[
 P_\ell(t)
   \le \Pbar_\ell(t)-\gamma
   < H-1+(2\ell-3)\gamma.
\]
Since $P_{\ell+1}$ has slope at most $1$, we deduce from
\eqref{app:lem6:eq} and the preceding estimate that, for
$t\in[q-1,q)$, we have
\[
 P_{\ell+1}(t)
  \ge P_{\ell+1}(q)-1
  \ge \Pbar_{\ell+1}(q)-2\gamma-1
  \ge H-1+(2\ell-2)\gamma
  > P_\ell(t)+\gamma.
\]
On the other hand, since $E_\ell$ is constant equal to $H$
on $[q,r)$, Lemma \ref{app:lem1} combined with
\eqref{app:lem6:eq} shows that, for any
$t\in[q,r)$, we have
\[
 P_\ell(t)+\gamma
   \le \Pbar_{\ell}(t)
   < H+2(\ell-1)\gamma
   \le \Pbar_{\ell+1}(q)-2\gamma
   \le \Pbar_{\ell+1}(t)-2\gamma
   \le P_{\ell+1}(t).
\]
Thus $P_{\ell+1}(t) > P_\ell(t)+\gamma$ for any
$t\in[q-1,r)$ and consequently, $M_\ell$ is
concave up on $[q-1,r]$.
\end{proof}

Our next lemma uses in a crucial way the hypothesis that
$\uP$ is reduced.

\begin{lemma}
 \label{app:lem7}
We have $\ell^{(i)}\ge k^{(i+1)}$ for each $i\ge 0$.
\end{lemma}

\begin{proof}
Suppose on the contrary that $\ell^{(i)}<k^{(i+1)}$
for some $i\ge 0$. For simplicity, we write $\ell:=\ell^{(i)}$
and $k':=k^{(i+1)}$.  Let
$q\in[c_i+n\gamma,d_i-n\gamma]$ be the single point of
discontinuity of $E_{\ell}$ on $[c_i,d_i]$ and let
$q'\in[c_{i+1}+n\gamma,d_{i+1}-n\gamma]$ be
the one of $E_{k'}$ on $[c_{i+1},d_{i+1}]$.
Since $\ell<k'$, we have $k'>1$ and $\ell<n$.
The definition of $k'=k^{(i+1)}$ in \eqref{app:eq:kl} yields
\[
 E_{k'-1}(q')
   \le E_{k'-1}(d_{i+1})
   < E_{k'}(d_{i+1})-1
   =E_{k'}(q')-1.
\]
Thus Lemma \ref{app:lem5} applies and shows that $M_{k'-1}$
is constant on $[q'-1,q']$.  As $\uP$ is $(n,\gamma)$-reduced,
and as $q'-n\gamma \ge c_{i+1}$, we deduce that $P_1,\dots,P_{k'-1}$
are constant on $[q'-1,c_{i+1}]$.  Note that this is an interval
of positive length because $c_{i+1} > d_{i+1}-1\ge q'-1$.  Since
$\ell<k'$, we conclude that $M_\ell=P_1+\cdots+P_\ell$ is
also constant on $[q'-1,c_{i+1}]$.

Similarly, we find that
\[
 E_{\ell+1}(q^-)
  \ge E_{\ell+1}(c_i)
  > E_{\ell}(c_i)+1
  = E_{\ell}(q^-)+1.
\]
Moreover, the function $E_\ell$ is constant on $[q,c_{i+1}]$
(because $\uE$ is constant on $[d_i,c_{i+1}]$).
Then, Lemma \ref{app:lem6} shows that $M_\ell$ is concave
up on $[q-1,c_{i+1}]$.  As it is constant on $[q'-1,c_{i+1}]$,
it must therefore be constant on the whole interval
$[q-1,c_{i+1}]$.  Using again the fact
that $\uP$ is $(n,\gamma)$-reduced, this implies that $P_\ell$
is constant on $[q-1,\,c_{i+1}-n\gamma]$.  This is a contradiction
because $q$ is an interior point of the latter interval (since
$q\le d_i-n\gamma< c_{i+1}-n\gamma$) and, by hypothesis, $P_\ell$ is
discontinuous at the point $q$.
\end{proof}

The next two lemmas complete the proof of Proposition \ref{app:prop}
when $\delta=1$.

\begin{lemma}
\label{app:lem8}
Define $t_i := E_1(c_i)+\cdots+E_n(c_i)$ for each $i\ge 0$.
Then, we have $t_0=n(n+1)/2$ and there exists a rigid integer $n$-system
$\uR\colon[t_0,\infty)\to\Delta_n$ such that $\uR(t_i)=\uE(c_i)$
for each $i\ge 0$.
\end{lemma}

\begin{proof}
Since $\uE$ is constant on $[0,c_0]$, we have $\uE(c_0)=\uE(0)
=(1,\dots,n)$ by \eqref{app:eq:E0} and thus $t_0=n(n+1)/2$.
For each $i\ge 0$, we partition $\Omega^{(i)}$ into maximal
subsets of consecutive integers $\{k,\dots,\ell\}$ such that
$E_k(c_i),\dots,E_\ell(c_i)$ are also consecutive integers and we
order these subsets in a sequence $\Omega_1^{(i)},\dots,
\Omega_{r_i}^{(i)}$ so that $\min\Omega_r^{(i)}>\max\Omega_{r+1}^{(i)}$
if $1\le r<r_i$.  Then we group all of these into a single
sequence
\[
 (\Omega_i)_{i\ge 0}
  =(\Omega_1^{(0)},\dots,\Omega_{r_0}^{(0)},
   \Omega_1^{(1)},\dots,\Omega_{r_1}^{(1)},\dots)
\]
and we define
\[
 k_i=\min\Omega_i, \quad \ell_{i+1}=\max\Omega_i
 \quad (i\ge 0).
\]
Clearly, we have $1\le k_i\le \ell_{i+1}\le n$ for each $i\ge 0$.
By definition of $k^{(i)}$ and $\ell^{(i)}$ in \eqref{app:eq:kl},
we also note that
\[
 k^{(i)}=\min\Omega_1^{(i)}
 \et
 \ell^{(i)}=\max\Omega_{r_i}^{(i)}
 \quad
 (i\ge 0).
\]
Thus the inequality $\ell^{(i)}\ge k^{(i+1)}$ of Lemma \ref{app:lem7}
translates into $\max\Omega_{r_i}^{(i)}\ge \min\Omega_1^{(i+1)}$ for each
$i\ge 0$.  Since we also have $\max\Omega_{r}^{(i)}>\min\Omega_{r+1}^{(i)}$
when $1\le r<r_i$, we conclude that $\max\Omega_i \ge \min\Omega_{i+1}$
for each $i\ge 0$, and so $\ell_i\ge k_i$ for each $i\ge 1$. Upon setting
$\ell_0=n$, the latter inequality extends to all $i\ge 0$.

For each $i\ge 0$ and each $r=1,\dots,r_i$, we define an integer
point $\ua_r^{(i)}=(a^{(i)}_{r,1},\dots,a^{(i)}_{r,n})$ by
\[
 a^{(i)}_{r,j}
  = \begin{cases}
     E_j(c_i) &\text{if \ $j\in\Omega_r^{(i)}\cup\cdots\cup\Omega_{r_i}^{(i)}$,}\\
     E_j(c_{i+1}) &\text{otherwise.}
    \end{cases}
\]
We also set $\ua_{r_i+1}^{(i)}=\ua_1^{(i+1)}$.  Then we form the sequence
\[
 (\ua^{(i)})_{i\ge 0}
 = (\ua_1^{(0)},\dots,\ua_{r_0}^{(0)},\ua_1^{(1)},\dots,\ua_{r_1}^{(1)},\dots).
\]
For fixed $i\ge 0$, we note that
\[
 E_j(c_{i+1})
 = \begin{cases}
      E_j(c_i)+1 &\text{if $j\in\Omega^{(i)}$,}\\
      E_j(c_i) &\text{otherwise}
   \end{cases}
 \quad
 (1\le j\le n),
\]
because the map $\uE$ is constant on $[d_i,c_{i+1}]$ and, by
Lemma \ref{app:lem3}, each of its component $E_j$ with
$j\in\Omega^{(i)}$ admits exactly one point of
discontinuity in $[c_i,d_i]$ while the other components
are constant on $[c_i,d_i]$.  This means that
\[
 \ua_1^{(i)}=\uE(c_i) \quad (i\ge 0)
\]
and that, for each $r=1,\dots,r_i$, the point $\ua_{r+1}^{(i)}$
is obtained from $\ua_r^{(i)}$ by adding $1$ to each of its
coordinates with index in $\Omega_r^{(i)}$.  Since these coordinates
are consecutive integers and since $\ua_{r_i+1}^{(i)}=\ua_1^{(i+1)}$,
we conclude that, in general, for each $i\ge 0$, the coordinates
of $\ua^{(i)}$ with index in $\Omega_i$
are consecutive integers and that $\ua^{(i+1)}$
is obtained from $\ua^{(i)}$ by adding $1$ to each of them.
As $\Omega_i=\{k_i,\dots,\ell_{i+1}\}$, this yields
\[
 (a^{(i)}_1,\dots,\widehat{a^{(i)}_{k_i}},\dots,a^{(i)}_n)
 = (a^{(i+1)}_1,\dots,\widehat{a^{(i+1)}_{\ell_{i+1}}},\dots,a^{(i+1)}_n)
 \et
 a^{(i+1)}_{\ell_{i+1}}=a^{(i)}_{\ell_{i+1}}+1
\]
for each $i\ge 0$.

The triple consisting of $(\ua^{(i)})_{i\ge 0}$, $(k_i)_{i\ge 0}$
and $(\ell_i)_{i\ge 0}$ is almost a canvas.  In view of the above,
it satisfies the conditions (C1) and (C3) of Definition
\ref{intro:def:canvas} and the slightly weaker
condition $1\le k_i\le\ell_i\le n$ ($i\ge 0$) instead of (C2).
Nevertheless, we can associate to it a map
$\uR\colon[t_0,\infty)\to\Delta_n$ as in Definition
\ref{intro:def:systems}.  By construction this map satisfies
$\uR(t_i)=\uE(c_i)$ for each $i\ge 0$ because each point
$\uE(c_i)$ belongs to the sequence $(\ua^{(i)})_{i\ge 0}$. Finally,
let $(i_m)_{0\le m<s}$ denote the sequence of integers $i\ge 0$
with $i=0$ or $k_i<\ell_i$, listed in increasing order.  We leave
to the reader to check that the triple $(\ua^{(i_m)})_{0\le m<s}$,
$(k_{i_m})_{0\le m<s}$, $(\ell_{i_m})_{0\le m<s}$ is a canvas and
that $\uR$ is the rigid system attached to it.
\end{proof}

\begin{lemma}
With the notation of the preceding lemma,
we have $\|\uR(t)-\uP(t)\|_\infty \le 3n^2$ for each
$t\in[t_0,\infty)$.
\end{lemma}

\begin{proof}
Fix a choice of $t\ge t_0$ and let $i\ge 0$
denote the index for which $t\in[t_i,t_{i+1})$.
Using Lemma \ref{app:lem1} (iii), we find
\[
 |c_i-t|
  \le |c_i-t_i| + t_{i+1}-t_i
   = \Big|c_i-\sum_{j=1}^n E_j(c_i)\Big|
      + \sum_{j=1}^n (E_j(c_{i+1})-E_j(c_i)) 
  \le n(n+1)+n.
\]
As the components of $\uP$ are continuous
and piecewise linear with slopes $0$ and $1$, we deduce that
\[
 \|\uP(c_i)-\uP(t)\|_\infty
  \le |c_i-t|
  \le n(n+2).
\]
By Lemma \ref{app:lem1} (ii), we also have
\[
 \|\uR(t_i)-\uP(c_i)\|_\infty
  = \|\uE(c_i)-\uP(c_i)\|_\infty
  \le n+1.
\]
Finally, since the components of $\uR$ are monotone increasing on
$[t_i,t_{i+1}]$, we find
\[
 \|\uR(t)-\uR(t_i)\|_\infty
 \le \|\uR(t_{i+1})-\uR(t_i)\|_\infty
 = \|\uE(c_{i+1})-\uE(c_i)\|_\infty
 = 1.
\]
Combining the three preceding displayed inequalities, we conclude that
$\|\uR(t)-\uP(t)\|_\infty \le n(n+2)+(n+1)+1\le 3n^2$.
\end{proof}

\section{Proof of Theorem \ref{intro:thm}}
\label{sec:proof}

Since any rigid $n$-system is an $(n,0)$-system, the first
result below proves the second assertion of our main
Theorem \ref{intro:thm} in a quantitative form.

\begin{theorem}
 \label{proof:thmA}
Let $q_0\ge 0$ and let $\uP\colon[q_0,\infty)\to\bR^n$ be an
$(n,0)$-system.  Then there exists a unit vector $\uu$
of $\bR^n$ such that $\|\uP(q)-\uL_\uu(q)\|_\infty\le 3n^2(n+9)$
for each $q\ge q_0$.
\end{theorem}

\begin{proof}
We first note that $\uP$ can be extended to an $(n,0)$-system
on $[0,\infty)$ in the following way.  Put $t_0=0$ and
$t_i=P_1(q_0)+\cdots+P_i(q_0)$ for $i=1,\dots,n$.  Then we have
$t_n=q_0$ and we define
\[
 \uP(q)=\Phi_n\big(0,\dots,0,P_1(q_0),\dots,P_{i-1}(q_0),q-t_{i-1}\big)
 \quad
 (t_{i-1}\le q\le t_i,\ 1\le i\le n),
\]
so that, for $i=1,\dots,n$ the combined graph of $\uP$ over
$[t_{i-1},t_i]$ consists of $n-1$ horizontal line segments,
not necessarily distinct, with ordinates $0,\dots,0,P_1(q_0),
\dots,P_{i-1}(q_0)$ and one line segment of slope $1$ joining
the points $(t_{i-1},0)$ and $(t_i,P_i(q_0))$.

In view of the above observation, we may assume that $q_0=0$.
Put $\delta=n+7$ and $\tilde{q}_0=n(n+1)\delta/2$.  Since
all $(n,0)$-systems are reduced, Proposition \ref{app:prop}
shows the existence of a rigid $n$-system
$\uR\colon[\tilde{q}_0,\infty)\to\bR^n$ of mesh $\delta$
such that $\|\uP(q)-\uR(q)\|_\infty\le 3n^2\delta$ for each
$q\ge \tilde{q}_0$.  For this rigid system, Theorem \ref{inv:thm}
shows in turn the existence of a unit vector $\uu$ in $\bR^n$
such that $\|\uR(q)-\uL_\uu(q)\|_\infty\le n\log(8e^4n)$ for each
$q\ge \tilde{q}_0$.  Then, we have
\begin{equation}
\label{proof:thmA:eq}
 \|\uP(q)-\uL_\uu(q)\|_\infty\le 3n^2\delta+n\log(8e^4n)
 \quad
 (q\ge \tilde{q}_0).
\end{equation}
However, for $q\in[0,\tilde{q}_0]$, the coordinates of $\uP(q)$
are non-negative and bounded above by $P_1(q)+\cdots+P_n(q)=q
\le \tilde{q}_0$ while those of $\uL_\uu(q)$ are also non-negative
and bounded above by $L_{\uu,1}(q)+\cdots+L_{n,q}(q)\le
\tilde{q}_0+n\log(n)$ thanks to \eqref{conv:eq:L1+...+Ln}.  Thus
the estimate \eqref{proof:thmA:eq} extends to all $q\ge 0$ and
the conclusion follows because $3n^2\delta+n\log(8e^4n)\le 3n^2(n+9)$.
\end{proof}

Since any rigid $n$-system of a given mesh $\delta>0$ is also a rigid
system of mesh $\delta/N$ for each integer $N\ge 1$, our last result
below implies the first assertion of Theorem \ref{intro:thm}, thereby
completing the proof of that theorem.

\begin{theorem}
 \label{proof:thmB}
Let $\delta > 24 n^4 2^n\log(n)$ and let $\uu$ be a unit vector of $\bR^n$.
Put $q_0=n(n+1)\delta/2$.  Then there exists a
rigid $n$-system $\uR\colon[q_0,\infty)\to\bR^n$ of mesh $\delta$
such that $\|\uL_\uu(q)-\uR(q)\|_\infty\le 4n^2\delta$ for each $q\ge q_0$.
\end{theorem}

\begin{proof}
Put $\gamma=6n2^n\log(n)$.  By Theorem \ref{conv:thm:SS}, there
exists an $(n,\gamma)$-system $\uP\colon[0,\infty)\to\bR^n$ such that
$\|\uL_\uu(q)-\uP(q)\|_\infty \le \gamma$ for each $q\ge 0$.  Then,
by Proposition \ref{red:prop}, there exists an $(n,2n\gamma)$-reduced
system $\tuP\colon[0,\infty)\to\bR^n$ for which
$\|\uP(q)-\tuP(q)\|_\infty \le n\gamma$ for each $q\ge 0$. Finally, since
$\delta>4n^3\gamma$, Proposition \ref{app:prop} provides a rigid
$n$-system $\uR\colon[q_0,\infty)\to\bR^n$ with mesh $\delta$ satisfying
$\|\tuP(q)-\uR(q)\|_\infty \le 3n^2\delta$ for each $q\ge q_0$. Thus,
for $q\ge q_0$, we get $\|\uL_\uu(q)-\uR(q)\|_\infty \le (n+1)\gamma+3n^2\delta
\le 4n^2\delta$.
\end{proof}

%%%%%%%%%%%%%%%%%%%%%%%%%%%%%%%%%%%%

\end{document}